\documentclass[a4paper, 11pt]{article}
\usepackage{graphicx}
\usepackage[T1]{fontenc}
\usepackage[utf8]{inputenc}
\usepackage{textcomp}
\usepackage[]{amsmath, amssymb}
\usepackage{fancyvrb}
\usepackage{caption}
\usepackage{float}
\usepackage{url}
\usepackage{pdfpages}
\usepackage{listings}
\usepackage{amsmath, amssymb}
\usepackage{amsthm}
\usepackage{grffile}
\usepackage{multirow}
\usepackage{pgfplots}
\usepackage{silence}
\usepackage{bm}
\usepackage{todonotes}
\usepackage{booktabs}
\usepackage{subcaption}

\usepackage{amsfonts}
\usepackage{algorithmic}
\ifpdf
  \DeclareGraphicsExtensions{.eps,.pdf,.png,.jpg}
\else
  \DeclareGraphicsExtensions{.eps}
\fi
\usepackage{mathtools}
\mathtoolsset{showonlyrefs}

\DefineVerbatimEnvironment{code}{Verbatim}{fontsize=\small, xleftmargin=5mm}
\newenvironment{remark}[1][Remark]{\begin{trivlist}
\item[\hskip \labelsep {\bfseries #1}]}{\end{trivlist}}

\DeclareMathOperator{\rank}{rank}

\newcommand{\R}{{\mathbb R}}

\newcommand{\C}{{\mathbb C}}

\newcommand{\dr}{\mathrm{d}}

\newcommand{\dd}[2]{\frac{\dr#1}{\dr#2}}
\newcommand{\pdd}[2]{\frac{\partial #1}{\partial #2}}

\newcommand{\bv}[1]{\mathbf{#1}}

\newcommand{\bfu}{\mathbf{u}}
\newcommand{\bfv}{\mathbf{v}}
\newcommand{\bfw}{\mathbf{w}}

\newcommand{\bff}{\mathbf{f}}
\newcommand{\bfg}{\mathbf{g}}

\newcommand{\bfx}{\mathbf{x}}

\newcommand{\bfphi}{\boldsymbol{\varphi}}
\newcommand{\bfpsi}{\boldsymbol{\psi}}
\newcommand{\bip}[3]{\left\langle #1,#2 \right\rangle_{#3}}
\newcommand{\bnorm}[2]{\langle\mkern-3mu\langle #1 \rangle\mkern-3mu\rangle_{#2}}
\newcommand{\ip}[3]{\left( #1,#2 \right)_{#3}}
\newcommand{\Or}{\mathcal{O}}
\newcommand{\dO}{\partial \Omega}

\newcommand{\norm}[1]{\|#1\|}
\newcommand{\seminorm}[1]{{\left\vert\kern-0.25ex\left\vert\kern-0.25ex\left\vert #1 
    \right\vert\kern-0.25ex\right\vert\kern-0.25ex\right\vert}}

\theoremstyle{plain}
\newtheorem{theorem}{Theorem}

\theoremstyle{definition}
\newtheorem{defi}{Definition}

\theoremstyle{remark}

\newlength\figureheight
\newlength\figurewidth

\definecolor{light-gray}{gray}{0.8}

\title{Order Preserving Interpolation for Summation-by-Parts Operators at Non-Conforming Grid Interfaces}
\author{Martin Almquist\thanks{Department of Geophysics, Stanford University, Stanford, CA 94305, United States. \texttt{malmquist@stanford.edu} }  \and Siyang Wang\thanks{Department of Mathematical Sciences, Chalmers University of Technology and University of Gothenburg, SE-412 96 Gothenburg,  
Sweden. \texttt{siyang.wang@chalmers.se} } \and Jonatan Werpers\thanks{Department of Information Technology, Uppsala University, SE-751 05 Uppsala, Sweden. \texttt{jonatan.werpers@it.uu.se} } }

\date{}

\begin{document}
\maketitle

\begin{abstract}
  We study non-conforming grid interfaces for summation-by-parts finite difference methods applied to partial differential equations with second derivatives in space. To maintain energy stability, previous efforts have been forced to accept a reduction of the global convergence rate by one order, due to large truncation errors at the non-conforming interface. We avoid the order reduction by generalizing the interface treatment and introducing order preserving interpolation operators. We prove that, given two diagonal-norm summation-by-parts schemes, order preserving interpolation operators with the necessary properties are guaranteed to exist, regardless of the grid-point distributions along the interface. The new methods retain the stability and global accuracy properties of the underlying schemes for conforming interfaces. 
\end{abstract}


\section{Introduction}
Adaptive mesh refinement is essential for efficiency in any simulation that requires high resolution in a localized area. For wave-dominated phenomena, high-order finite difference (FD) methods are often computationally efficient, but not always robust. By combining summation-by-parts (SBP) operators \cite{KreissScherer74} with simultaneous approximation terms (SATs) \cite{CarpenterGottlieb94}, the SBP-SAT methodology leads to energy stable and conservative high-order FD methods on multi-block and curvilinear grids \cite{DelReyFernandez2014a,Svard2014}.

Mattsson and Carpenter \cite{MattssonCarpenter09} extended the SBP-SAT framework to locally refined grids by constructing SBP preserving interpolation operators for non-conforming grid interfaces. They proved energy stability for conservation laws and parabolic equations. The approach has since been extended to the Schrödinger equation \cite{Nissen2012a}, the second order wave equation \cite{Wang2016b}, and the advection-diffusion equation \cite{Lundquist2018}. For equations with second derivatives in space, it has been observed that the SATs at the non-conforming interfaces worsen the largest local truncation error, and hence the global convergence rate, by one order, as compared to conforming interfaces \cite{Lundquist2018,Nissen_15,Wang2017int}. The obvious remedy would have been to increase the order of accuracy of the interpolation operators, but \cite{Lundquist2015} showed that this is impossible because the order of the interpolation operators is bounded from above by the order of the quadrature rule associated with every SBP operator \cite{Hicken13}. Following this discovery, \cite{Friedrich2017} constructed interpolation operators corresponding to novel degree-preserving first derivative SBP operators that are based on extra-high order quadrature rules. Unlike traditional SBP FD operators, the boundary closures of the degree-preserving operators depend on the number of grid points. It is not obvious how well this approach would extend to equations with second derivatives and variable coefficients considering the significant effort involved in constructing such SBP operators \cite{Mattsson11}.

In this work we circumvent the order reduction for second order PDEs without increasing the quadrature order. Instead, we introduce new interpolation operators referred to as order-preserving (OP) interpolation operators. The key property of the OP operators is that they come in two pairs. While each pair suffers from the accuracy restriction derived in \cite{Lundquist2018,Lundquist2015}, it is possible to avoid accuracy reduction at the interface by using particular SATs. We prove a theoretical result, which states that given two diagonal-norm SBP operators, it is always possible to construct matching OP interpolation operators. Encouraged by this result, we construct OP operators for the special case of 2:1 grid refinement and perform numerical experiments. With these new interpolation operators, our experiments show the same global convergence rates as if the interfaces were conforming for the heat equation, the Schr\"{o}dinger equation, and the second order wave equation. 

SBP preserving interpolation operators are not only used for coupling FD grid blocks of equal sizes. In \cite{Kozdon2014}, they were used to couple FD and discontinuous Galerkin methods, and in \cite{Lundquist2018} it was shown that they are the key to coupling arbitrary SBP schemes on general meshes. In \cite{Nissen_15}, interpolation operators for so called T-junctions, i.e.\ grids that are non-conforming at the block level, were constructed. A recent study also coupled finite difference and finite element methods \cite{Gao2018}. Although we only present experiments with FD methods, the ideas in this paper apply to hybrid methods.

The numerical experiments in \cite{Nissen_15} showed quite erratic convergence behaviour when interpolation operators were used for the Schrödinger equation. Although the average rate was approximately as expected, the rate from one refinement level to the next varied noticeably. Numerical experiments in this paper show that erratic convergence behaviour for the Schrödinger equation can be mitigated by using the same discretization of the Laplacian as for the second order wave equation.

The paper is organized as follows. We introduce some notation in 
section \ref{sec:prel}. In section \ref{sec_schrodinger_heat} we introduce the OP operators for the Schrödinger and heat equations. In section \ref{sec:int_op} we prove that, given two SBP operators based on diagonal quadrature rules, OP interpolation operators always exist. Using OP operators, we derive an order-preserving coupling for the second order wave equation in section \ref{sec:wave}. We then discuss three different discretizations of the Laplacian and their properties in section \ref{sec_Laplacian}. In section \ref{sec:num_exp}, we present numerical experiments with the Schrödinger, heat, and second order wave equations. The OP coupling exhibits one order higher global convergence compared with previous approaches using the interpolation operators derived in \cite{MattssonCarpenter09}. We conclude in section \ref{sec:conclusion}.

\section{Preliminaries} \label{sec:prel}
In this section we introduce some notation and recall the properties of second derivative SBP operators that will be needed in subsequent sections. 

Let $\Omega$ denote a bounded domain in $\R^2$ and let $f = [f_1, \ldots, f_m]^T$ and $g = g_1, \ldots, g_m]^T$ be vector-valued functions with $m$ components in $C^2(\Omega)$. We will use the standard inner product and norm on $L^2(\Omega)$, i.e.
\begin{equation} \label{eq:inner_prod_continuous}
  \ip{f}{g}{\Omega} = \int_{\Omega} \! f^*g \, \mathrm{d} \Omega, \quad \norm{f}^2_{\Omega} = \ip{f}{f}{\Omega} ,
\end{equation}
where $^*$ denotes conjugate transpose. For integration along the boundary $\partial \Omega$, we use the notation
\begin{equation}
  \bip{f}{g}{\partial \Omega} = \int_{\partial \Omega} \! f^* g \, \mathrm{d} \Gamma .
\end{equation}
We will frequently use the normal derivative on $\partial \Omega$, defined by 
\begin{equation}
\pdd{f_i}{\hat{n}} = \nabla f_i \cdot \hat{n} ,
\end{equation}
where $\hat{n}$ denotes the outward unit normal. With this notation, Green's first identity reads
\begin{equation} \label{eq:green_cont}
  \ip{f_i}{\Delta g_i}{\Omega} = \bip{f_i}{\pdd{g_i}{\hat{n}}}{\dO} - \ip{\nabla f_i}{\nabla g_i}{\Omega} .
\end{equation}

Let $U = [-1,0]\times[0,1]$ and $V = [0,1]\times[0,1]$. We denote the interface between $U$ and $V$ by $\Gamma$.  The outer boundaries are $\Gamma_U = \partial U \setminus \Gamma$ and $\Gamma_V = \partial V \setminus \Gamma$. Throughout this paper we will consider initial-boundary value problems in the form
\begin{equation} \label{eq:general}
\begin{array}{rll}
\smallskip
L_U u = 0, &(x,y) \in U, & t \in [0, \, T], \\
\smallskip
L_V v = 0,  &(x,y) \in V, & t \in [0, \, T], \\
\smallskip
u - v = 0,  &(x,y) \in \Gamma, & t \in [0, \, T], \\
\smallskip
\alpha_U \pdd{u}{\hat{n}_U} + \alpha_V \pdd{v}{\hat{n}_V}=0,  &(x,y) \in \Gamma, & t \in [0, \, T], \\
\smallskip
B_U u = 0,  &(x,y) \in \Gamma_U, & t \in [0, \, T], \\
\smallskip
B_V v = 0,  &(x,y) \in \Gamma_V, & t \in [0, \, T], \\
\end{array}
\end{equation}
augmented with suitable initial conditions at time $t=0$. Here $L_{U,V}$ are linear differential operators of second order in space and first or second order in time; $\hat{n}_{U,V}$ are outward unit normals on the interface $\Gamma$; $\alpha_{U,V}$ are scalar coefficients; and $B_{U,V}$ are boundary operators. Because our focus is on the interface treatment, we assume that the boundary operators $B_{U,V}$ are such that the problems are well-posed and will henceforth omit them in the analysis.

\subsection{Summation-by-parts finite difference operators}
Consider an interval $J = [x_{\ell}, \, x_{r}]$ and a grid vector $\mathbf{x} = [x_1, \ldots, x_N]^T$ that discretizes $J$. Let $F,G \in C^{\infty}(J)$ and let $F({\mathbf{x})} = [F(x_1),\ldots,F(x_N)]^T$ denote the restriction of $F$ to $\mathbf{x}$. Consider a difference operator $D_2$ that approximates the second derivative, i.e.
\begin{equation}
    D_2 F(\mathbf{x}) \simeq F''(\mathbf{x}).
\end{equation}
We say that $D_2$ has the summation-by-parts (SBP) property \cite{MattssonNordstrom04} if it can be decomposed as 
\begin{equation}
    D_2 = H^{-1}(-\widetilde D^T H \widetilde D + e_{r}d_{r}^T - e_{\ell}d_{\ell}^T ),
\end{equation}
where $H=H^T > 0$; $e_{\ell,r}^T$ are row vectors that interpolate to the left and right boundaries of the domain; $d_{\ell,r}^T$ are row vectors that approximate the first derivatives at the domain boundaries; and $F(\bv{x})^T\widetilde D^T H \widetilde D G(\bv{x})$ approximates $\int_J F'G' \, \mathrm{d}x$. Note that although the operator $\widetilde D$ is accurate in the integrated sense, $\widetilde{D}$ by itself is not necessarily an accurate approximation of the first derivative. We further require that
\begin{equation} \label{eq:borrowing}
    \widetilde D^T H \widetilde D = h \gamma (d_r d_r^T + d_{\ell} d_{\ell}^T) + \widetilde{M},
\end{equation}
where $\gamma$ is a positive constant and $\widetilde{M}$ is symmetric positive semidefinite. The property in \eqref{eq:borrowing} is essential for the inter-block coupling in the second order wave equation \cite{MattssonHam08}.

Let $\Omega$ denote the unit square. We assume for ease of notation that $\Omega$ is discretized by a tensor-product grid with $N$ grid points in each coordinate direction. Let $D_2$ denote an SBP operator corresponding to a single grid line and let $I$ denote the $N \times N$ identity matrix. On $\Omega$, we will use the following operators:
\begin{align}
    D_{xx} &= D_2 \otimes I,                       & D_{yy} &= I \otimes D_2 , \\
    \widetilde D_x &= \widetilde D \otimes I,      & \widetilde D_y &= I \otimes \widetilde D  \\
    D_{\Delta} &= D_{xx} + D_{yy} ,                & H_{\Omega} &= H \otimes H \\
    \widetilde{M}_x &= \widetilde{M} \otimes H,    & \widetilde{M}_y &= H \otimes \widetilde{M} \\
    e_{W,E} &= e_{\ell,r} \otimes I,               &  e_{S,N} &= I \otimes e_{\ell,r} , \\
    d_{W} &= -d_{\ell} \otimes I,                  &  d_{S} &= -I \otimes d_{\ell} , \\
    d_{E} &= d_{r} \otimes I,                      &  d_{N} &= I \otimes d_{r}. \\
\end{align}
Note the minus signs in the definitions of $d_{W,S}$, which ensure that all boundary derivative operators $d_{W,S,E,N}$ approximate the outward normal derivative.

Let $\mathbf{f}$ and $\mathbf{g}$ denote the restrictions of $m$-valued functions $f=[f_1,\ldots,f_m]^T$ and $g=[g_1,\ldots,g_m]^T$ to the grid. We store the discrete values in $N^2 \times m$ matrices,
\begin{equation}
  \bff = \begin{bmatrix} \bff_1 & \bff_2 & \ldots & \bff_m \end{bmatrix}, \quad \bfg = \begin{bmatrix} \bfg_1 & \bfg_2 & \ldots & \bfg_m \end{bmatrix},
 \end{equation}
 where $\bff_i$ denotes the restriction of $f_i$ to the grid. We define the operator $\widetilde D_{\nabla}$ by
\begin{equation}
  \widetilde D_{\nabla} \bff_i = \begin{bmatrix} \widetilde D_{x} \bff_i & \widetilde D_{y} \bff_i  \end{bmatrix}.
\end{equation}
  We further define a discrete inner product in $\C^{N^2 \times m}$, and a corresponding norm,
 \begin{equation} \label{eq:inner_prod_discrete}
  (\bff, \bfg) = \sum \limits_{i=1}^m \bff_i^* H_{\Omega} \bfg_i , \quad \norm{\bff}^2 = (\bff,\bff) .
 \end{equation}
Notice that \eqref{eq:inner_prod_discrete} approximates \eqref{eq:inner_prod_continuous}. For boundary integrals we use the notation
\begin{equation} \label{eq:discr_boundary_integral}
  \bip{e_W^T \bff_i}{e_W^T \bfg_i}{ } = (e_W^T \bff_i)^* H (e_W^T \bfg_i) ,
\end{equation}
and
\begin{equation}
  \bnorm{e_W^T \bff_i}{ }^2 = \bip{e_W^T \bff_i}{e_W^T \bff_i}{ } .
\end{equation}
Because of the SBP properties, the discrete operators satisfy a discrete version of Green's identity \eqref{eq:green_cont},
\begin{equation} \label{eq:ibp_discr_2D}
\begin{aligned}
  \ip{\bff_i}{D_{\Delta} \bfg_i}{} &=
  -\ip{\widetilde D_{\nabla} \bff_i}{\widetilde D_{\nabla} \bfg_i}{}
  +\sum \limits_{\alpha \in \{ W,S,E,N \} } \bip{e_{\alpha}^T \bff_i}{d_{\alpha}^T \bfg_i}{}. \\
  &
\end{aligned}
\end{equation}
Because of the property in \eqref{eq:borrowing} we have
\begin{equation} \label{eq:borrowing_2D}
  \ip{\widetilde D_{\nabla} \bff_i}{\widetilde D_{\nabla} \bfg_i}{} = \bff_i^* (\widetilde{M}_x + \widetilde{M}_y) \bfg_i + \sum \limits_{\alpha \in \{ W,S,E,N \} } h \gamma \bip{d_{\alpha}^T \bff_i}{d_{\alpha}^T \bfg_i}{},
\end{equation}
where $\widetilde{M}_{x,y}$ are symmetric positive semidefinite.

\subsection{Notation for two-block discretizations}
Consider the two-block problem \eqref{eq:general}.
Let $\bfu$ and $\bfv$ approximate $u$ and $v$, respectively. In each grid block we introduce SBP finite difference operators in the way described in the previous section. The operators could be based on different numbers of grid points or different orders of accuracy and are hence not identical, in general. For notational convenience however, we will not use different symbols for the differential operators, e.g.\ $D_{\Delta}$. But in this setting it is important to distinguish between the different quadratures that are used in the two blocks. Let $H_U^{\Gamma}$ denote the one-dimensional quadrature matrix $H$ used in $U$ and let $H_V^{\Gamma}$ denote the quadrature matrix in $V$. The corresponding two-dimensional quadrature matrices are
\begin{equation}
  H_U = H_U^{\Gamma} \otimes H_U^{\Gamma}, \quad H_V = H_V^{\Gamma} \otimes H_V^{\Gamma}.
\end{equation}
We will write $\ip{\cdot}{\cdot}{U}$ for the inner product used in $U$ and $\ip{\cdot}{\cdot}{V}$ for the inner product in $V$. Similarly, the boundary integrals corresponding to \eqref{eq:discr_boundary_integral} will be denoted by $\bip{\cdot}{\cdot}{U}$ and $\bip{\cdot}{\cdot}{V}$. Using the summation-by-parts fomula \eqref{eq:ibp_discr_2D} for $U$ and $V$ we can write
\begin{equation} \label{eq:sbp_property_2d_u}
  \ip{\bfu}{D_{\Delta} \boldsymbol{\phi}}{U} = \bip{e_E^T \bfu}{d_E^T \boldsymbol{\phi}}{U} - \ip{\widetilde D_{\nabla} \bfu}{\widetilde D_{\nabla} \boldsymbol{\phi}}{U},
\end{equation}
and
\begin{equation} \label{eq:sbp_property_2d_v}
  \ip{\bfv}{D_{\Delta} \boldsymbol{\psi}}{V} = \bip{e_W^T \bfv}{d_W^T \boldsymbol{\psi}}{V} - \ip{\widetilde D_{\nabla} \bfv}{\widetilde D_{\nabla} \boldsymbol{\psi}}{V}.
\end{equation}
Here, we have ignored all boundary terms not belonging to the interface between $U$ and $V$ since they do not enter the stability analysis of the interface treatment. They should be taken care of by proper enforcement of well-posed boundary conditions.

Similarly, the property in \eqref{eq:borrowing_2D} leads to
\begin{equation} \label{eq:borrowing_u}
  \ip{\widetilde D_{\nabla} \bfu}{\widetilde D_{\nabla} \bfphi}{U} = h_u \gamma_u \bip{d_E^T \bfu}{d_E^T \bfphi}{U} + \bfu^* \left( \widetilde{M}_x + \widetilde{M}_y \right) \bfphi ,
\end{equation}
and
\begin{equation} \label{eq:borrowing_v}
  \ip{\widetilde D_{\nabla} \bfv}{\widetilde D_{\nabla} \bfpsi}{V} = h_v \gamma_v \bip{d_W^T \bfv}{d_W^T \bfpsi}{V} + \bfv^* \left( \widetilde{M}_x + \widetilde{M}_y \right) \bfpsi ,
\end{equation}
where $h_{u,v}$ denote the grid spacings and the constants $\gamma_{u,v}$ correspond to the (possibly different) SBP operators used in $U$ and $V$.

\section{Parabolic and Schrödinger type equations}\label{sec_schrodinger_heat}

In this section we consider initial-boundary value problems of the form
\begin{equation} \label{eq:SE_heat_2D_continuous}
\begin{array}{rll}
\smallskip
u_t - a \Delta u = 0, & (x,y) \in U, & t \in [0, \,T], \\
\smallskip
v_t - b \Delta v = 0, & (x,y) \in V, & t \in [0, \,T], \\
\smallskip
u -  v = 0, & (x,y) \in \Gamma, & t \in [0, \,T], \\
\smallskip
a\pdd{u}{\hat{n}_U} + b\pdd{v}{\hat{n}_V} = 0, & (x,y) \in \Gamma, & t \in [0,\, T] , \\
\end{array}
\end{equation}
with initial data for $u$ and $v$. We assume that $a$ and $b$ are constant coefficients with non-negative real parts. If $a$ and $b$ are real and positive, \eqref{eq:SE_heat_2D_continuous} is a parabolic problem. Purely imaginary $a$ and $b$ yield a Schrödinger-type problem. 

Before discretizing \eqref{eq:SE_heat_2D_continuous}, we derive an energy estimate for the continuous problem. Multiplying the first equation in \eqref{eq:SE_heat_2D_continuous} by $u^*$ and integrating over $U$ yields
\begin{equation}
(u,u_t)_{U} = a(u,\Delta u)_{U} = a \bip{u}{\pdd{u}{\hat{n}_U}}{\partial U} - a \norm{\nabla u}^2_{U} = a \bip{u}{\pdd{u}{\hat{n}_U}}{\Gamma} - a \norm{\nabla u}^2_{U}, 
\end{equation}
where we discarded terms related to outer boundaries in the last step. By repeating the procedure on $V$ and using the interface conditions, we obtain the estimate
\begin{equation}
    \begin{aligned}
      \dd{}{t} \left( \norm{u}^2_{U} + \norm{v}^2_{V} \right)
      =&- (a+a^*) \norm{\nabla u}^2_{U} - (b+b^*) \norm{\nabla v}^2_{V}.
    \end{aligned}
\end{equation} 

\subsection{Semi-discrete approximation}
The semi-discrete approximation of \eqref{eq:SE_heat_2D_continuous} can be written as
\begin{equation} \label{eq:scheme_schr_heat}
\begin{aligned}
\bfu_t - a D_{\Delta}\bfu &=  SAT_u, \\
\bfv_t - b D_{\Delta}\bfv &=  SAT_v, \\
\end{aligned}
\end{equation}
where $SAT_{u,v}$ are penalty terms that weakly impose the interface conditions on $\Gamma$. For notational convenience we write e.g.\ $\bfu_t$ for $\dd{\bfu}{t}$. 
We make the ansatz
\small
\begin{equation}
\begin{aligned}
  SAT_u &= \tau_u a^* H_U^{-1} d_E H_U^{\Gamma} (e_E^T \bfu - I_{v2u}^{e} e_W^T \bfv) + \sigma_u H_U^{-1} e_E H_U^{\Gamma} (ad_E^T \bfu + I_{v2u}^{d} b d_W^T \bfv), \\
  SAT_v &= \tau_v b^* H_V^{-1} d_W H_V^{\Gamma} (e_W^T \bfv - I_{u2v}^{e} e_E^T \bfu) + \sigma_v H_V^{-1} e_W H_V^{\Gamma} (bd_W^T \bfv + I_{u2v}^{d} a d_E^T \bfu), \\
\end{aligned}
\end{equation}
\normalsize
where $\tau_{u,v}$ and $\sigma_{u,v}$ are scalar penalty parameters. The interpolation operators $I_{v2u}^{e,d}$ and $I_{u2v}^{e,d}$ interpolate between the two different grids that discretize the interface $\Gamma$. The superscripts $e$ and $d$ specify what the interpolation operator is applied to; either the solution itself ($e$), or the normal derivative ($d$). In the case of matching grids (and matching quadratures), all the interpolation operators can be replaced by identity matrices. Note that the above ansatz is more general than those used for the Schrödinger equation in \cite{Nissen2012a} and parabolic equations in \cite{MattssonCarpenter09}. The extra generality is what will enable us to obtain higher order of accuracy. The above ansatz reduces to the ones in \cite{MattssonCarpenter09,Nissen2012a} if we set
\begin{equation}
  I_{v2u}^{e} = I_{v2u}^{d}, \quad I_{u2v}^{e} = I_{u2v}^{d} .
\end{equation}
In the following subsection we analyze the local truncation errors introduced by $SAT_{u,v}$ in \eqref{eq:scheme_schr_heat}, to determine what accuracies we require of the different interpolation operators. After that, we derive stability conditions on the interpolation operators.

\subsection{Local truncation errors and convergence rates}
Let $\bv{x}_u$ and $\bv{x}_v$ be two different grid vectors that discretize an interval $J$. Let $f$ be a smooth function and let $\bff_{u}$ and $\bff_{v}$ denote its restrictions to $\bfx_u$ and $\bfx_v$, respectively. If $I_{u2v}$ and $I_{v2u}$ are interpolation operators of orders $q_{u2v}$ and $q_{v2u}$, then
\begin{equation}\label{int_err}
\begin{aligned}
I_{u2v} \bff_{u} &= \bff_{v} + \Or(h^{q_{u2v}}) , \\
I_{v2u} \bff_{v} &= \bff_{u} + \Or(h^{q_{v2u}}) ,
\end{aligned}
\end{equation}
where the interpolation error in \eqref{int_err} denotes the maximum error over all grid points. 


Traditional diagonal-norm SBP operators \cite{Mattsson11,MattssonNordstrom04} with 2$p$th order of accuracy in the interior are of order $p$ at a fixed number of near-boundary grid points. In two dimensions, all grid points near an interface or boundary are affected by the $p$th order errors. Still, for equations with second derivatives in space, numerical experiments often show $\min(2p,p+2)$th order convergence rates. For one dimensional problems, a general normal mode analysis shows that the convergence rate of pointwise stable schemes is at least $\min(2p,p+2)$ \cite{SvardNordstrom06}. A detailed analysis of a class of SBP--SAT discretizations for the second order wave equation proves that, with properly chosen penalty parameter, the convergence rate sometimes exceeds $\min(2p,p+2)$ \cite{Wang2017}. This result is extended to two dimensional problems in \cite{Wang2018}. However, when conforming grid interfaces are present the rate does not exceed $\min(2p,p+2)$, in general. We will refer to $\min(2p,p+2)$ as the ideal rate that we hope to preserve when using interpolation operators at non-conforming interfaces.  

The boundary derivative operators $d_{W,E,S,N}^T$ are typically constructed to be of order $p+1$ to ensure that SATs at boundaries and conforming interfaces do not cause truncation errors larger than $p$th order. Hence, we can ignore their effect in this discussion. To analyze the local truncation errors of the SATs, we let $\bfw_{u,v}$ denote the restrictions of the exact solution to the grids on the left and right sides of the interface. Assume that e.g.\ $I_{u2v}^e$ is of order $q_{u2v}^e$. The truncation errors are
\begin{align}
  T_u =\ & \tau_u a^* \underbrace{H_U^{-1} d_E H_U^{\Gamma}}_{\sim h^{-2}} \underbrace{(e_E^T \bfw_u - I_{v2u}^{e} e_W^T \bfw_v)}_{\sim h^{q_{v2u}^e}} \\ 
  &+ \sigma_u \underbrace{H_U^{-1} e_E H_U^{\Gamma}}_{\sim h^{-1}} \underbrace{(a d_E^T \bfw_u + I_{v2u}^{d} b d_W^T \bfw_v)}_{\sim h^{q_{v2u}^d}} \\
  =\ & \Or(h^{q_{v2u}^e - 2}) + \Or(h^{q_{v2u}^d-1}),
\end{align}
and
\begin{equation}
\begin{aligned}
  T_v =\ & \tau_v b^* \underbrace{H_V^{-1} d_W H_V^{\Gamma}}_{\sim h^{-2}} \underbrace{(e_W^T \bfw_v - I_{u2v}^{e} e_E^T \bfw_u)}_{\sim h^{q_{u2v}^e }} \\
   &+ \sigma_v \underbrace{H_V^{-1} e_W H_V^{\Gamma}}_{\sim h^{-1}} \underbrace{(b d_W^T \bfw_v + I_{u2v}^{d} a d_E^T \bfw_u)}_{\sim h^{q_{v2u}^d }} \\
  =\ & \Or(h^{q_{u2v}^e - 2}) + \Or(h^{q_{u2v}^d-1}) .
\end{aligned}
\end{equation}
Note that to balance the errors from the different interpolation operators we require
\begin{equation}
  q_{u2v}^e = q_{v2u}^e = q_{u2v}^d + 1 = q_{v2u}^d + 1 .  
\end{equation}
That is, $I_{u2v}^e$ and $I_{v2u}^e$ should ideally be one order more accurate than $I_{u2v}^d$ and $I_{v2u}^d$. 

Suppose that diagonal-norm SBP operators of interior order $2p$ are used on both sides of the interface. The inner product matrices of traditional SBP finite difference operators of interior accuracy $2p$ correspond to quadrature rules of order $2p$ \cite{Hicken13}. As will be discussed in more detail in section \ref{sec:int_op}, the order of the inner product matrix is what limits the accuracy of the interpolation operators. The highest achievable accuracy turns out to be
\begin{equation}
  q_{u2v}^e = q_{v2u}^e = p+1, \quad q_{u2v}^d = q_{v2u}^d = p,
\end{equation}
which leads to 
\begin{equation}
  T_{u,v} = \Or(h^{p-1}) .
\end{equation}
Recall that the diagonal-norm SBP derivative operators have accuracy $p$ at all grid points along the interface. Ideally we would have wanted $T_{u,v} = \Or(h^p)$, but since that is impossible, the best we can do is to construct interpolation operators such that $T_{u,v} = \Or(h^{p-1})$ only at $\Or(1)$ grid points along the interface. It is not obvious how the global convergence rate will be affected by a localized large truncation error. Because the error is more localized than that of the derivative operators, one would expect at least $(p+1)$th order global convergence, but could hope for higher order. Indeed, we observe $(p+2)$th order global convergence for $p=2,3$ in the numerical experiments in section \ref{sec:num_exp}.

In the less general interface coupling with
\begin{equation}
  I_{u2v}^e = I_{u2v}^d = I_{u2v} , \quad I_{v2u}^e = I_{v2u}^d = I_{v2u} ,
\end{equation}
we obtain local truncation errors
\begin{equation}
  T_u = \Or(h^{q_{v2u}-2}) , \quad T_v = \Or(h^{q_{u2v}-2}) .
\end{equation}
The stability requirements on the interpolation operators limit the accuracies according to (see \cite{Lundquist2018})
\begin{equation}
  q_{u2v} + q_{v2u} \leq 2p+1 .
\end{equation}
In this case, it is inevitable that $\max(T_u,T_v) = \Or(h^{p-2})$. Hence, with the new OP coupling, the largest local truncation error is of one order higher than the largest local truncation error in previous approaches. Therefore, it is reasonable to expect an improvement by one order in global convergence rate.

\subsection{Stability}

The aim in this subsection is to derive stability conditions on the interpolation operators. We first introduce the notation
\begin{equation}
  I_{u2v}^g := I_{u2v}^e, \quad I_{u2v}^b := I_{u2v}^d , \quad I_{v2u}^g := I_{v2u}^e, \quad I_{v2u}^b := I_{v2u}^d ,  
\end{equation}
where the superscripts $g$ and $b$ denote ``good'' and ``bad''. The error analysis in the previous subsection showed that all the SATs in \eqref{eq:scheme_schr_heat} give truncation errors of equal order if the good interpolation operators are one order more accurate than the bad ones.

To analyze stability we multiply the first and second equations in \eqref{eq:scheme_schr_heat} by $\bfu^* H_U$ and $\bfv^* H_V$, respectively,  which leads to
\begin{equation}
\begin{aligned}
  \ip{\bfu}{\bfu_t}{U} &= a \ip{\bfu}{D_{\Delta} \bfu }{U} + \ip{\bfu}{SAT_u}{U} , \\
  \ip{\bfv}{\bfv_t}{V} &= b \ip{\bfv}{D_{\Delta} \bfv }{V} + \ip{\bfv}{SAT_v}{V} .
\end{aligned}
\end{equation}
Using the summation-by-parts formulas \eqref{eq:sbp_property_2d_u} and \eqref{eq:sbp_property_2d_v}, we obtain
\begin{equation}
\begin{aligned}
  a \ip{\bfu}{D_{\Delta} \bfu }{U} &= a\bip{e_E^T \bfu}{d_E^T \bfu }{U} - a \ip{\widetilde D_{\nabla} \bfu}{\widetilde D_{\nabla} \bfu }{U} \\
  &= a\bip{e_E^T \bfu}{d_E^T \bfu }{U} - a \norm{\widetilde D_{\nabla} \bfu}^2_{U} ,
\end{aligned}
\end{equation}
and similarly
\begin{equation}
  b \ip{\bfv}{D_{\Delta} \bfv }{V} = b \bip{e_W^T \bfv}{d_W^T \bfv }{V} - b \norm{\widetilde D_{\nabla} \bfv}^2_{V} .
\end{equation}
The SATs yield
\begin{equation}
(\bfu,SAT_u)_{U} = \tau_u a^* \langle d_E^T \bfu, e_E^T \bfu - I_{v2u}^{g} e_W^T \bfv \rangle_{U} + \sigma_u \langle e_E^T \bfu, a d_E^T \bfu + I_{v2u}^{b} b d_W^T \bfv \rangle_{U} ,
\end{equation}
\begin{equation}
(\bfv,SAT_v)_{V} = \tau_v b^* \langle d_W^T \bfv, e_W^T \bfv - I_{u2v}^{g} e_E^T \bfu \rangle_{V} + \sigma_v \langle e_W^T \bfv, b d_W^T \bfv + I_{u2v}^{b} a d_E^T \bfu \rangle_{V} .
\end{equation}
We can now conclude that the discrete energy rate is
\begin{equation}
\begin{aligned}
\dd{}{t} \left( \norm{\bfu}^2_{U} + \norm{\bfv}^2_{V} \right) = - (a+a^*) \norm{\widetilde D_{\nabla} \bfu}_U^2 - (b+b^*) \norm{\widetilde D_{\nabla} \bfv}_V^2 + w^* A w,
\end{aligned}
\end{equation}
where we have defined
\begin{equation}
  w = \begin{bmatrix} 
e_E^T \bfu \\
e_W^T \bfv \\
d_E^T \bfu \\
d_W^T \bfv \\
\end{bmatrix}
, \quad
A = \begin{bmatrix}
& & \alpha_{13} & \alpha_{14} \\
& & \alpha_{23} & \alpha_{24} \\
\alpha_{13}^* & \alpha_{23}^* & & \\
\alpha_{14}^* & \alpha_{24}^* & & \\
\end{bmatrix},
\end{equation}
and
\begin{equation}
\begin{aligned}
  \alpha_{13} &= \left( a + a\tau_u^* + a\sigma_u \right)H_U^{\Gamma}, \\
  \alpha_{14} &= b\sigma_u H_U^{\Gamma} I_{v2u}^b - b\tau_v^* (I_{u2v}^{g})^* H_V^{\Gamma} , \\ 
  \alpha_{23} &= -a\tau_u^* (I_{v2u}^g)^* H_U^{\Gamma} + a\sigma_v H_V^{\Gamma} I_{u2v}^b, \\ 
  \alpha_{24} &= (b + b\tau_v^* + b\sigma_v) H_V^{\Gamma} .
\end{aligned}
\end{equation} 
The matrix $A$ is symmetric and has zeros on the diagonal. To ensure that $w^* A w$ is non-positive we need all the elements of $A$ to vanish.
When the grids are conforming and the same SBP operators are used in $U$ and $V$, all interpolation operators can be replaced by identity matrices and the inner products are the same, i.e.\ $H_U^{\Gamma} = H_V^{\Gamma}$. The stability conditions then reduce to
\begin{equation} \label{eq:parameter_conditions}
\begin{aligned}
1 + \tau_u^* + \sigma_u &= 0 ,\\
\sigma_u - \tau_v^* &= 0 ,\\
-\tau_u^* + \sigma_v &= 0 ,\\
1 + \tau_v^* + \sigma_v &= 0 ,
\end{aligned}
\end{equation}
which is equivalent to
\begin{equation}
\begin{aligned}
  1 + \tau_u^* + \tau_v^* &= 0 ,\\
\sigma_u  &= \tau_v^* ,\\
\sigma_v &= \tau_u^* .
\end{aligned}
\end{equation}
There is one free parameter, but the only solution that treats the left and right directions identically is \cite{Berg2012heat,Nissen2012a}
\begin{equation} \label{eq:parameter_values}
  \tau_u = \sigma_u = \tau_v = \sigma_v = -1/2.
\end{equation}
Now we return to the general case of non-conforming grids. If the interpolation operators satisfy
\begin{equation} \label{eq:stability_cond_matrix}
\begin{aligned}
H_U^{\Gamma} I_{v2u}^b &= (I_{u2v}^{g})^* H_V^{\Gamma},\\
(I_{v2u}^g)^* H_U^{\Gamma} &= H_V^{\Gamma} I_{u2v}^b,
\end{aligned}
\end{equation}
then the condition that all entries of $A$ equal zero again reduces to \eqref{eq:parameter_conditions} and the parameter values in \eqref{eq:parameter_values} yield stability. The stability condition \eqref{eq:stability_cond_matrix} relates $I_{v2u}^b$ to $I_{u2v}^g$ and $I_{u2v}^b$ to $I_{v2u}^g$. Thus, we may use two pairs of operators that are unrelated to one another.  

\subsection{The stability condition in terms of Hilbert adjoints}
Let $N_U$ and $N_V$ denote the number of grid points along the interface in $U$ and $V$, respectively. The matrices $H_U^{\Gamma}$ and $H_V^{\Gamma}$ define inner products in $\C^{N_U}$ and $\C^{N_V}$ by
\begin{equation}
  \ip{\bfu_{\Gamma}}{\boldsymbol{\phi}_{\Gamma}}{U_{\Gamma}} =\bfu_{\Gamma}^* H_{U}^{\Gamma} \boldsymbol{\phi}_{\Gamma} , \quad \ip{\bfv_{\Gamma}}{\boldsymbol{\psi}_{\Gamma}}{V_{\Gamma}} =\bfv_{\Gamma}^* H_{V}^{\Gamma} \boldsymbol{\psi}_{\Gamma} .
\end{equation}
 Let the resulting inner product spaces be denoted by $U_{\Gamma}$ and $V_{\Gamma}$. It follows from the completeness of $\C$ that also $U_{\Gamma}$ and $V_{\Gamma}$ are complete, and hence Hilbert spaces. By definition, the Hilbert adjoint $L^{\dagger}$ of a linear operator $L:U_{\Gamma} \mapsto V_{\Gamma}$ satsifies
\begin{equation} \label{eq:def_hilbert}
  (v,Lu)_{V_{\Gamma}} = (L^{\dagger}v,u)_{U_{\Gamma}} \quad \forall \, u \in U_{\Gamma}, v \in V_{\Gamma} .
\end{equation}
Linear operators from $\C^m$ to $\C^n$ are represented by rectangular matrices and hence \eqref{eq:def_hilbert} is equivalent to the condition
\begin{equation}
  H_V^{\Gamma} L = (L^{\dagger})^* H_U^{\Gamma} .
\end{equation}
We note that the stability condition \eqref{eq:stability_cond_matrix} can be equivalently written as 
\begin{equation} \label{eq:stability_cond_adjoint}
\begin{aligned} 
I_{v2u}^b &= (I_{u2v}^{g})^{\dagger}, \\
I_{v2u}^g &= (I_{u2v}^b)^{\dagger} .
\end{aligned}
\end{equation}
To obtain a stable scheme, it is enough to choose two interpolation operators, say $I_{v2u}^g$ and $I_{u2v}^g$. The two remaining operators $I_{u2v}^b$ and $I_{v2u}^b$ are uniquely determined as the adjoints of the first two operators. But for the scheme to be accurate, both $I_{v2u}^g$ and $I_{u2v}^g$, as well as their adjoints, must be accurate interpolation operators. Naturally, the stability condition \eqref{eq:stability_cond_adjoint} is also essential for dual-consistent (or adjoint-consistent) discretizations with non-conforming grids.

Let $q(I)$ denote the order of accuracy of an interpolation operator $I$. Based on the stability and accuracy analysis in this section, we introduce the following definition.
\begin{defi} \label{def:op}
 Given two inner product matrices $H_U^{\Gamma}$ and $H_V^{\Gamma}$ that correspond to quadrature rules of order $2p$, we say that the interpolation operators $I_{u2v}^{g}$, $I_{u2v}^{b}$, $I_{v2u}^{g}$, and $I_{v2u}^{b}$ constitute a set of order preserving interpolation operators if
\begin{equation} 
\begin{aligned} 
I_{v2u}^b &= (I_{u2v}^{g})^{\dagger}, \\
I_{v2u}^g &= (I_{u2v}^b)^{\dagger} ,
\end{aligned}
\end{equation}
and
\begin{equation}
\begin{aligned} 
q \left( I_{u2v}^{g} \right) &= q \left( I_{v2u}^{g} \right) = p+1, \\ 
q \left( I_{u2v}^{b} \right) &= q \left( I_{v2u}^{b} \right) = p.
\end{aligned}
\end{equation}
\end{defi}
The order preserving (OP) operators are defined so that the scheme \eqref{eq:scheme_schr_heat} is stable with truncation error of order $p-1$ in maximum norm.

\section{Existence of interpolation operators} \label{sec:int_op}
In this section we will restate the known results that bound the sum of the orders of the interpolation operators $I_{u2v}$ and $I_{v2u} = I_{u2v}^{\dagger}$ from above. As an example, in the case of traditional SBP operators with $2p$th order interior stencils on both sides of the interface, the bound is 
\begin{equation} \label{eq:bound}
q(I_{u2v}) + q(I_{v2u}) \leq 2p+1.
\end{equation}
It is important to note that the sum of the orders is an odd number. When using only one adjoint pair of interpolation operators, the global order will be dictated by $\min(q(I_{u2v}), q(I_{v2u}))$, which can not exceed $p$. Hence, previous approaches \cite{Lundquist2018,MattssonCarpenter09,Nissen_15,Wang2017int} have not had a reason to let $\max(q(I_{u2v}), q(I_{v2u})) = p+1$. The OP approach with two pairs of operators utilizes the extra order to improve the global convergence rate.

After restating the known results we present a new theorem that shows that the bounds similar to \eqref{eq:bound} are always sharp. That is, one can always construct an adjoint pair $I_{u2v} = I_{v2u}^{\dagger}$ with the maximal accuracy allowed by the bounds. Further, the total order of accuracy may be divided arbitrarily between $I_{u2v}$ and $I_{v2u}$. Guided by the new existence result, we proceed to construct new OP interpolation operators for the special case of a 2:1 grid size ratio and $2p$th order interior stencils on both sides, for $2p=2,4,6,8$.

\subsection{Theoretical results}
Consider two vectors $\bfx_u = [x_{1},...,x_{N_u}]^T$ and $\bfx_v  = [\xi_{1},...,\xi_{N_v}]^T$ that discretize an interval $J = [\alpha,\beta]$. For monomials $x^j , \, j\geq 0$, we write e.g.\ $\bfx_u^j = [x^j_{1},...,x^j_{N_u}]^T $. Consider two inner product matrices $H_{u,v}$. In this section we use the inner product notation
\begin{equation}
  (\bfu, \bfphi)_u = \bfu^* H_u \bfphi , \quad (\bfv, \bfpsi)_v = \bfv^* H_v \bfpsi.
\end{equation}
We also introduce Vandermonde-like matrices
\begin{equation}
  X_u^{m,n} = [\bfx_{u}^{m},\bfx_{u}^{m+1},\ldots,\bfx_u^n].
\end{equation}
Assume that the inner product matrices $H_{u,v}$ correspond to quadrature rules of orders $q_{u,v}$ on $\bfx_{u,v}$. This means that $H_{u,v}$ integrate polynomials of degree less than $q_{u,v}$ exactly, i.e.,
\begin{equation} \label{eq:quad_deg_u}
  \ip{\bfx_u^i}{\bfx_u^j}{u} = \frac{1}{i+j+1} (\beta^{i+j+1} - \alpha^{i+j+1}) , \quad i+j < q_u ,
\end{equation}
and similarly for $H_v$. We also assume that the orders are not in fact higher than $q_{u,v}$, i.e.,
\begin{equation}
  \ip{\bfx_u^i}{\bfx_u^j}{u} \neq \frac{1}{i+j+1} (\beta^{i+j+1} - \alpha^{i+j+1}) , \quad i+j = q_u  ,
\end{equation}
with similar conditions for $H_v$. In the case $q_u = q_v = q$, we assume that $H_u$ and $H_v$ do not have the same leading order error, i.e.,
\begin{equation} \label{eq:assumption_trunc_error}
    \ip{\bfx_u^i}{\bfx_u^j}{u} \neq \ip{\bfx_v^i}{\bfx_v^j}{v} , \quad i+j = q .
\end{equation}
This condition is justifiable in practice. Consider for example the case when $H_u$ and $H_v$ are based on the same quadrature formula but with different numbers of grid points, e.g.\ $N_u > N_v$. Then we expect $H_u$ to have a smaller truncation error than $H_v$ and hence the leading order error terms will not be equal.

The order conditions \eqref{eq:quad_deg_u} imply that
\begin{equation} \label{eq:deg_vandermonde_uv}
  (X_u^{0,i})^T H_u X_u^{0,j} = (X_v^{0,i})^T H_v X_v^{0,j} , \quad i + j = \min(q_u,q_v) - 1 .
\end{equation}
An interpolation operator $I_{u2v}$ is accurate of order $q$ if it is exact for polynomials of degree up to $q-1$, i.e.
\begin{equation}
  I_{u2v} X_u^{0,q-1} = X_v^{0,q-1} .
\end{equation}
The first important theorem is due to Lundquist et al.\ \cite{Lundquist2018}.
\begin{theorem} \label{thm:max_order}
If $I_{v2u} = I_{u2v}^{\dagger}$, then
\begin{equation}
  q(I_{u2v}) + q(I_{v2u}) \leq \min(q_u,q_v) + 1 .
\end{equation}
\end{theorem} 
\begin{proof}
See \cite{Lundquist2018}.
\end{proof}
A proof of Theorem \ref{thm:max_order} for the special case $q_u=q_v$ first appeared in \cite{Lundquist2015}.

The following theorem shows that given two inner product matrices, it is always possible to construct an adjoint pair $I_{v2u} = I_{u2v}^{\dagger}$ with the maximal accuracy allowed by Theorem \ref{thm:max_order}.
\begin{theorem} \label{thm:existence}
Let $q_{u2v}$ and $q_{v2u}$ be integers such that $1 \leq q_{u2v} \leq N_u$, $1 \leq q_{v2u} \leq N_v$ and
\begin{equation}
  q_{u2v} + q_{v2u} \leq \min(q_u,q_v) + 1.
\end{equation}
 Then there exists $I_{u2v}$ such that we may set $I_{v2u} = I_{u2v}^{\dagger}$ and obtain
\begin{equation}
  q(I_{u2v}) = q_{u2v}, \quad q(I_{v2u}) = q_{v2u} .
\end{equation}
\end{theorem}
\begin{proof}
We seek $I_{u2v}$ such that $q(I_{u2v}) = q_{u2v}$, i.e.,
\begin{equation} \label{eq:cond_vander_1}
  I_{u2v} X_u^{0, \,q_{u2v}-1} = X_v^{0, \,q_{u2v}-1} .
\end{equation}
Upon setting $I_{v2u} = I_{u2v}^{\dagger}$, the condition that $q(I_{v2u}) = q_{v2u}$ yields
\begin{equation} 
  X_u^{0, \, q_{v2u}-1 } = I_{v2u} X_v^{0, \, q_{v2u}-1 } = I_{u2v}^{\dagger} X_v^{0, \, q_{v2u} -1} = H_u^{-1} I_{u2v}^{T} H_v X_v^{0, \, q_{v2u} -1} ,
\end{equation}
which we may write as
\begin{equation} \label{eq:cond_vander_2}
  I_{u2v}^{T} H_v X_v^{0, \, q_{v2u}-1 } = H_u X_u^{0, \, q_{v2u}-1 } .
\end{equation}
Thus, we are seeking an operator $I_{u2v}$ that satsfies the two accuracy conditions \eqref{eq:cond_vander_1} and \eqref{eq:cond_vander_2}.
At this point, we require the assumption that $N_u \geq q_{u2v}$, i.e.\ that there are sufficiently many grid points. We organize the remainder of the proof into two parts. First, we show that the result holds in the case $N_u = q_{u2v}$. Second, we prove the result for $N_u > q_{u2v}$.

Assume that $N_u=q_{u2v}$. In this case, $X_u^{0,q_{u2v}-1}=V_u^{q_{u2v}}$, where $V_u^{n}$ denotes the Vandermonde matrix of order $n$. The Vandermonde matrix is square and invertible, so the first accuracy condition \eqref{eq:cond_vander_1} determines $I_{u2v}$ uniquely:
\begin{equation}
  I_{u2v} = X_v^{0, \,q_{u2v}-1} (V_u^{q_{u2v}})^{-1}  .
\end{equation}
Substituting the expression for $I_{u2v}$ in the second condition \eqref{eq:cond_vander_2} yields
\begin{equation}
  (V_u^{q_{u2v}})^{-T} (X_v^{0,\, q_{u2v}-1})^{T} H_v X_v^{0,\,q_{v2u}-1} = H_u X_u^{0,\,q_{v2u}-1}, 
\end{equation}
which is equivalent to
\begin{equation} \label{eq:order_cond_quad}
  (X_v^{0,\,q_{u2v}-1})^{T} H_v X_v^{0,\,q_{v2u}-1} = (X_u^{0,\, q_{u2v}-1})^{T} H_u X_u^{0,\,q_{v2u}-1} .
\end{equation}
If \eqref{eq:order_cond_quad} is satisfied, then $I_{u2v}$ and $I_{u2v}^{\dagger}$ are accurate of orders $q_{u2v}$ and $q_{v2u}$. But \eqref{eq:order_cond_quad} is satisfied for any $q_{u2v}$, $q_{v2u}$ such that $q_{u2v} + q_{v2u}\leq \min(q_u,q_v) + 1$, because both quadrature matrices integrate such polynomials exactly. 

It remains to prove that we can find $I_{u2v}$ with the desired accuracy properties when $N_u > q_{u2v}$. To obtain an invertible matrix in the left-hand side of \eqref{eq:cond_vander_1} we can pad the system with extra equations where the right-hand side is arbitrary, i.e.,
\begin{equation} \label{eq:acc_vander_padded}
  I_{u2v} V_u^{N_u} = \widetilde{X}_v^{q_{u2v}-1}= \begin{bmatrix} X_{v}^{0,\,q_{u2v}-1} & \widetilde{X}_v \end{bmatrix} ,
\end{equation}
where the entries of $\widetilde{X}_v$ are arbitrary. The columns of $\widetilde{X}_v$ are the results of applying $I_{u2v}$ to polynomials of degree larger than or equal to $q_{u2v}$, which we do not need to put any conditions on. By solving \eqref{eq:acc_vander_padded} for $I_{u2v}$ and substituting in the second condition \eqref{eq:cond_vander_2}, we arrive at the system
\begin{equation} \label{eq:vandermonde_tilde}
  (\widetilde{X}_v^{q_{u2v}-1})^{T} H_v X_v^{0,\, q_{v2u}-1} = (V_u^{N_u})^{T} H_u X_u^{0,\, q_{v2u}-1}, 
\end{equation}
and we must prove that there exists $\widetilde{X}_v$ that satisfies this equation. Using the block structure of $\widetilde{X}_v^{q_{u2v}-1}$ and $V_u^{N_u}$, \eqref{eq:vandermonde_tilde} can be written as
\begin{equation}
  \begin{bmatrix} (X_{v}^{0,\,q_{u2v}-1})^T H_v X_v^{0,\, q_{v2u}-1}\\ \widetilde{X}_v^T H_v X_v^{0,\, q_{v2u}-1}\end{bmatrix}
   =
  \begin{bmatrix}
    (X_u^{0, \, q_{u2v}-1})^{T}H_u X_u^{0,\, q_{v2u}-1} \\ (X_u^{q_{u2v}, \,N_u-1})^{T} H_u X_u^{0,\, q_{v2u}-1} 
  \end{bmatrix} .
\end{equation}
The upper system of equations is again satisfied if $q_{u2v} + q_{v2u} \leq \min(q_u,q_v)+1$ because both quadrature matrices integrate such polynomials exactly. It remains to show that we can find $\widetilde{X}_v$ that satisfies
\begin{equation}
  \widetilde{X}_v^T H_v X_v^{0,\, q_{v2u}-1} = (X_u^{q_{u2v}, \,N_u-1})^{T} H_u X_u^{0,\, q_{v2u}-1} .
\end{equation}
Let $A=(X_v^{0,\, q_{v2u}-1})^{T} H_v$ and $B = (X_u^{0,\, q_{v2u}-1})^{T} H_u X_u^{q_{u2v}, \,N_u-1}$ so that the system can be written as
\begin{equation} \label{eq:lin_sys_eq_thm3_AB}
  A \widetilde{X}_v = B ,
\end{equation} 
where $A \in \R^{q_{v2u} \times N_v}$. A sufficient condition for a solution matrix $\widetilde{X}_v$ to exist is that $A$ has full row rank. Because 
\begin{equation}
  \rank(H_v) = N_v , \quad \rank(X_v^{0,\, q_{v2u}-1}) = \min(q_{v2u},N_v) ,
\end{equation}
it follows that
\begin{equation}
  \rank(A) = \min(q_{v2u},N_v) .
\end{equation}
Hence, $A$ has full row rank if $N_v \geq q_{v2u}$. This proves the theorem.
\end{proof}

\begin{remark}
Let $\min(q_u,q_v)=2p$. By Theorem \ref{thm:max_order}, $q(I_{u2v}) + q(I_{v2u}) \leq 2p+1$. Theorem \ref{thm:existence} guarantees that interpolation operator pairs such that $q(I_{u2v}) + q(I_{v2u}) = 2p+1$ exist.
Further, it shows that we may distribute the total order of $2p+1$ freely. In this paper, we are only interested in operators such that
\begin{equation}
  q(I_{u2v}) = p+1 , \quad q(I_{v2u}) = p,
\end{equation}
or
\begin{equation}
  q(I_{u2v}) = p , \quad q(I_{v2u}) = p+1 ,
\end{equation}
because these choices lead to a balance of truncation errors in the numerical scheme. 
However, operator pairs such that e.g.\
\begin{equation}
  q(I_{u2v}) = p+2 , \quad q(I_{v2u}) = p-1 ,
\end{equation}
are also guaranteed to exist.
\end{remark}

\begin{remark}
Note that Theorem \ref{thm:existence} only concerns the interpolation error in maximum norm. For e.g.\ finite volume, finite element and discontinuous Galerkin methods, this is all that matters. However, to obtain the ideal convergence rate with traditional finite difference methods we require the interpolation operators to have smaller $\ell^2$-errors than the point-wise errors guaranteed by Theorem \ref{thm:existence}. That is, the interpolation operators should be accurate of orders $q_{u2v}$ and $q_{v2u}$ at $\Or(1)$ grid points only, and at least one order more accurate at remaining grid points. Theorem \ref{thm:existence} does not guarantee that such operators exist. In our experience however, it is not difficult to obtain high order for the interior grid points where the quadrature weights are constant.
\end{remark}

\subsection{Examples of order preserving interpolation operators}

For the numerical experiments in section \ref{sec:num_exp}, we have constructed OP interpolation operators for the special case of a 2:1 grid size ratio and $2p$th order interior stencils on both sides, for $2p=2,4,6,8$. They are compatible with diagonal-norm SBP operators with minimal number of boundary points on equidistant grids, see e.g.\ \cite{strand94}. Actually, they are compatible with any SBP operator based on the same norm matrix. The norm matrices $H_{2p}$ are
\begin{equation}
\begin{aligned}
  H_2 &= h \, \mbox{diag}\left( \left[ \tfrac{1}{2} \; 1 \; \cdots  \right] \right), \\ 
  H_4 &= h \, \mbox{diag}\left( \left[ \tfrac{17}{48} \; \tfrac{59}{48} \; \tfrac{43}{48} \; \tfrac{49}{48} \; 1 \; \cdots  \right] \right), \\
  H_6 &= h \, \mbox{diag}\left( \left[ \tfrac{13649}{43200} \; \tfrac{12013}{8640} \; \tfrac{2711}{4320} \; \tfrac{5359}{4320} \; \tfrac{7877}{8640} \; \tfrac{43801}{43200} \; 1 \; \cdots  \right] \right), \\
  H_8 &= h \, \mbox{diag}\left( \left[ \tfrac{1498139}{5080320} \; \tfrac{1107307}{725760} \; \tfrac{20761}{80640} \; 
            \tfrac{1304999}{725760} \; \tfrac{299527}{725760} \; \tfrac{103097}{80640} \; \tfrac{670091}{725760} \;
            \tfrac{5127739}{5080320} \; 1 \; \cdots  \right] \right). \\
\end{aligned}
\end{equation}
 As prescribed in section \ref{def:op}, the constructed OP operators satisfy
 $q(I_{u2v}^g)=q(I_{v2u}^g)=p+1$ and $q(I_{u2v}^b)=q(I_{v2u}^b)=p$. These large truncation errors are localized to the boundary closures, which are comparable in size to the boundary closures of the corresponding difference operators. All constructed interpolation operators have repeating stencils of order $2p$ in the interior. 

When constructing the OP operators, we made an ansatz that $I_{u2v}$ is sparse, with non-zero boundary blocks of size $m \times n$ and an interior bandwidth $d$. 
The accuracy conditions of $I_{u2v}$ and $I_{u2v}^{\dagger}$ were required to be fulfilled exactly. If the linear system that results from the accuracy conditions did not have a solution, $m$, $n$ and $d$ were successively increased until the system became solvable. Any remaining free parameters were then used to minimize the $\ell^2$ error when interpolating a sine function with 8 grid points per wavelength on the coarse grid. Figure \ref{fig:operator_figures} shows the sparsity pattern for the $p=2$ case. The OP operators are available at
\url{https://bitbucket.org/martinalmquist/op_interpolation_operators}.

\begin{figure}[h]
    \begin{subfigure}[b]{.5\linewidth}
        \centering
        \includegraphics[width=4.5cm]{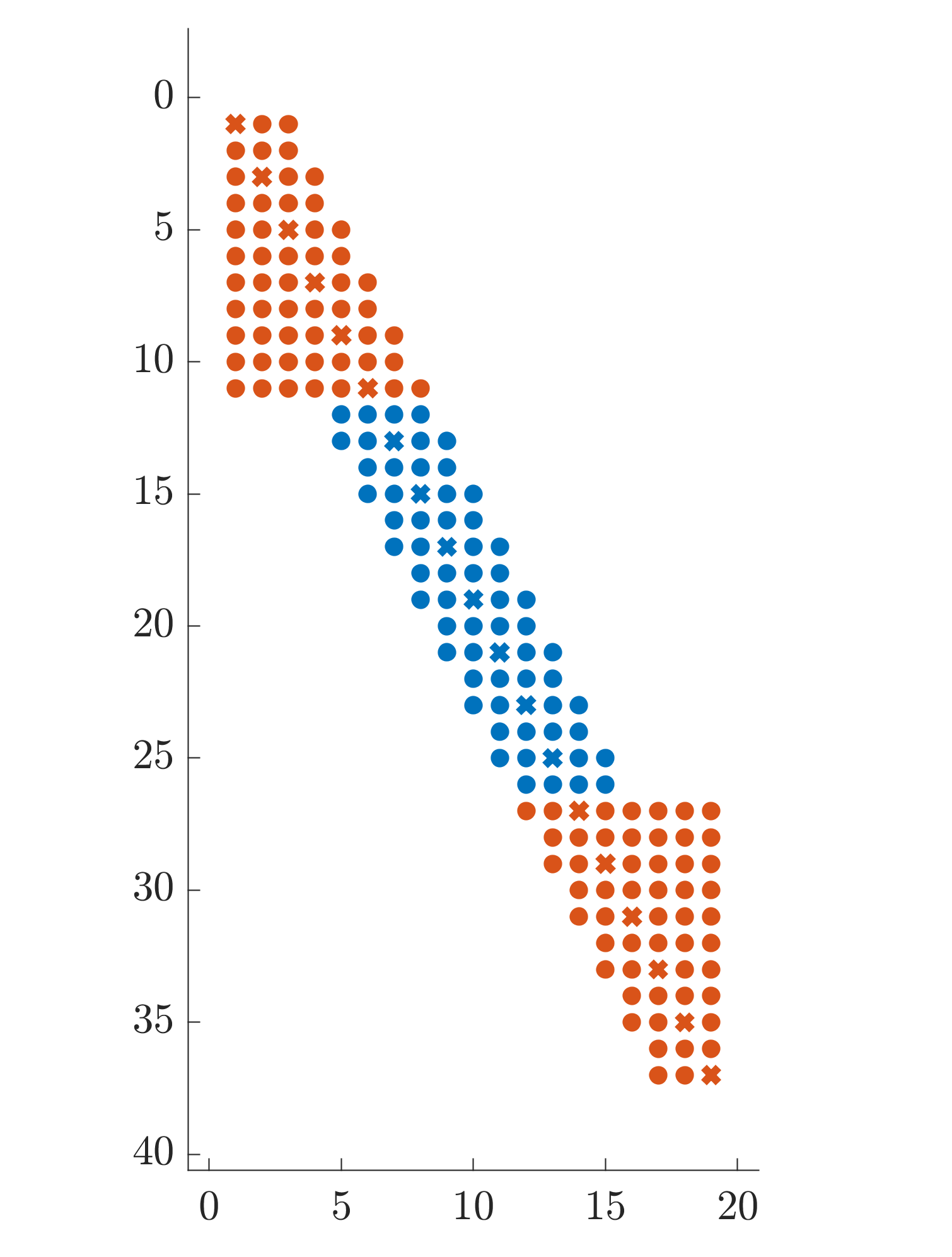}
        \caption{$I^g_{u2v}$}\label{fig:Icf}
    \end{subfigure}%
    \begin{subfigure}[b]{.5\linewidth}
        \centering
        \includegraphics[height=4.5cm]{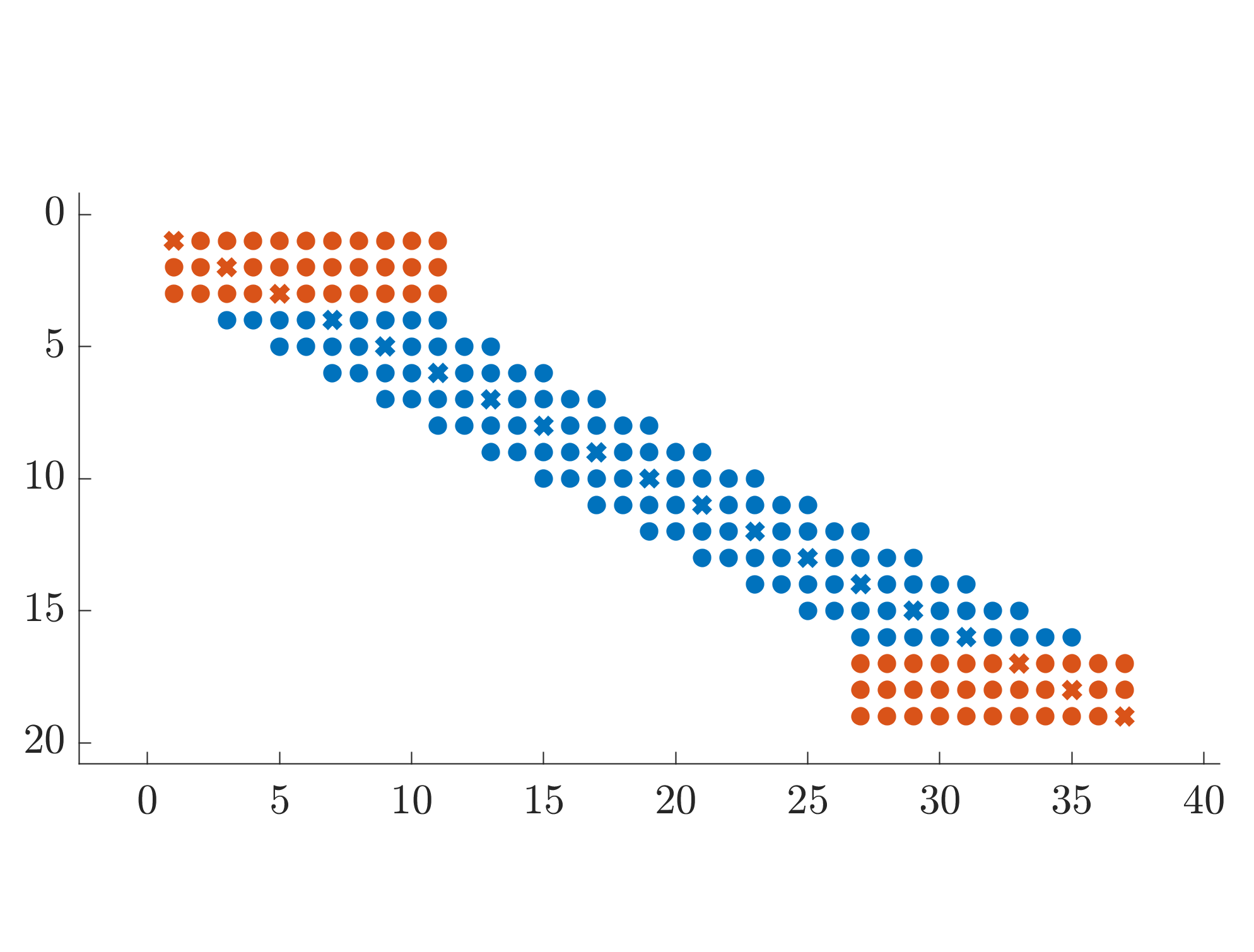}
        \caption{$I^b_{v2u}$}\label{fig:Ifc}
    \end{subfigure}
    \caption{Non-zero elements for an example pair of interpolation operators ($p=2$). The grid ratio is 2 and the coarse grid has 20 points. Modified boundary blocks are shown in red.}
    \label{fig:operator_figures}
\end{figure}

\section{The second order wave equation} \label{sec:wave}

In this section we will use OP interpolation operators to derive stable and accurate discretizations of the wave equation with non-conforming interfaces. The new schemes are similar to those derived in \cite{Wang2017int}, but with truncation errors of order $p-1$ instead of $p-2$. Since numerical experiments in \cite{Wang2017int} showed global convergence rates of order $p+1$, one may expect the new schemes to converge with rate $p+2$, which is supported by numerical experiments in section \ref{sec:num_exp}.

We consider the problem
\begin{equation} \label{eq:wave_2D_continuous}
\begin{array}{rll}
\smallskip
u_{tt} - c_1^2 \Delta u = 0, & (x,y) \in U, & t \in [0, \,T], \\
\smallskip
v_{tt} - c_2^2 \Delta v = 0, & (x,y) \in V, & t \in [0, \,T], \\
\smallskip
u -  v = 0, & (x,y) \in \Gamma, & t \in [0, \,T], \\
\smallskip
c_1^2 \pdd{u}{\hat{n}_U} + c_2^2 \pdd{v}{\hat{n}_V} = 0, & (x,y) \in \Gamma, & t \in [0,\, T] , \\
\end{array}
\end{equation}
augmented with initial data for $u$, $u_t$, $v$, and $v_t$. We assume that $c_{1,2}$ are real, positive constants. The problem \eqref{eq:wave_2D_continuous} is energy-conserving and satisfies
\begin{equation}
  \dd{}{t} E_{wave} = 0 ,
\end{equation}
where
\begin{equation} \label{eq:energy_cont_wave}
  E_{wave} = \frac{1}{2} \left(\norm{u_t}^2_{U} + c_1^2 \norm{\nabla u}^2_{U} +   \norm{v_t}^2_{V} + c_2^2 \norm{\nabla v}^2_{V} \right) .
\end{equation}

\subsection{Semi-discrete approximation}
The semi-discrete approximation of \eqref{eq:wave_2D_continuous} can be written as
\begin{equation} \label{eq:scheme_wave}
\begin{aligned}
\bfu_{tt} - c_1^2 D_{\Delta} \bfu &= SAT_u, \\
\bfv_{tt} - c_2^2 D_{\Delta} \bfv &= SAT_v, \\
\end{aligned}
\end{equation}
where $SAT_{u,v}$ are penalty terms that weakly impose the interface conditions on $\Gamma$. We here make the ansatz
\footnotesize
\begin{equation}
\begin{aligned}
    SAT_u = &-H_U^{-1} \left[
        \frac{\tau_u}{h_u} c_1^2 e_E H_U^{\Gamma}(e_E^T \bfu - I_{v2u}^{g} e_W^T \bfv)
        + \frac{\sigma_u}{h_v} c_2^2 e_E H_U^{\Gamma} I_{v2u}^{b} (I_{u2v}^g e_E^T \bfu - e_W^T \bfv)
    \right] \\
    &+ H_U^{-1} \left[
        \frac{c_1^2}{2} d_E H_U^{\Gamma} (e_E^T \bfu - I_{v2u}^{g} e_W^T \bfv)
        - \frac{1}{2} e_E H_U^{\Gamma} (c_1^2 d_E^T \bfu + c_2^2 I_{v2u}^{b} d_W^T \bfv)
    \right] \\
    SAT_v = &- H_V^{-1} \left[
        \frac{\tau_v}{h_v} c_2^2 e_W H_V^{\Gamma}(e_W^T \bfv - I_{u2v}^{g} e_E^T \bfu)
        + \frac{\sigma_v}{h_u} c_1^2 e_W H_V^{\Gamma} I_{u2v}^{b} (I_{v2u}^g e_W^T \bfv - e_E^T \bfu)
    \right] \\
    &+ H_V^{-1} \left[
        \frac{c_2^2}{2} d_W H_V^{\Gamma} (e_W^T \bfv - I_{u2v}^{g} e_E^T \bfu)
        - \frac{1}{2} e_W H_V^{\Gamma} (c_2^2 d_W^T \bfv + c_1^2 I_{u2v}^{b} d_E^T \bfu)
    \right] .
\end{aligned}
\end{equation}
\normalsize
Assuming that the penalty parameters $\tau_{u,v}$ and $\sigma_{u,v}$ all are $\Or(1)$, all the SATs give rise to local truncation errors that are $\Or(h^{p-1})$.

Compared to the SBP-SAT method for conforming interfaces \cite{MattssonHam08}, there is an additional penalty term on each side of the interface. The second term in $SAT_u$ evaluates the residual of the condition $u=v$ on the grid at the $v$-side of the interface, and then uses $I_{v2u}^b$ to interpolate the residual to the $u$ grid. We point out that the order of accuracy of $I_{v2u}^b$ does not affect the order of the local truncation error of this term. The second term in $SAT_v$ is analogous. Notice that in the case of conforming grids we may replace all interpolation operators by the identity matrix, in which case the above ansatz reduces to the one used in \cite{MattssonHam08}.

\begin{theorem} \label{thm:wave}
The scheme \eqref{eq:scheme_wave} is stable if $\sigma_v = \tau_u = \frac{\theta_u}{4\gamma_u}$ and $\sigma_u = \tau_v = \frac{\theta_v}{4\gamma_v}$, where $\theta_u, \theta_v \geq 1$.
\end{theorem}
\begin{proof}
Multiplying the first equation in \eqref{eq:scheme_wave} by $\bfu_t^T H_U$ yields
\begin{equation}\label{ut_utt}
\begin{aligned}
    \ip{\bfu_t}{\bfu_{tt}}{U} = \; & c_1^2 \ip{\bfu_t}{D_{\Delta} \bfu}{U}
    - \frac{\tau_u}{h_u} c_1^2 \bip{e_E^T \bfu_t}{e_E^T \bfu - I_{v2u}^g e_W^T \bfv}{U} \\
   &- \frac{\sigma_u}{h_v} c_2^2 \bip{e_E^T \bfu_t}{I_{v2u}^b (I_{u2v}^g e_E^T \bfu - e_W^T \bfv)}{U} \\
   &+ \frac{c_1^2}{2} \bip{d_E^T \bfu_t}{e_E^T \bfu - I_{v2u}^g e_W^T \bfv}{U} \\
   &- \frac{1}{2} \bip{e_E^T \bfu_t}{c_1^2 d_E^T \bfu + c_2^2 I_{v2u}^b d_W^T \bfv}{U} . \\
\end{aligned}
\end{equation}
We now rewrite the terms one by one. By the SBP properties of $D_{\Delta}$, we have 
\begin{equation}
  c_1^2 \ip{\bfu_t}{D_{\Delta} \bfu}{U} = c_1^2 \bip{e_E^T \bfu_t}{d_E^T \bfu}{U} - c_1^2 \ip{\widetilde D_{\nabla} \bv{u}_t}{\widetilde D_{\nabla} \bv{u}}{U} .
\end{equation}
We write the first boundary integral in \eqref{ut_utt} as the sum of two integrals,
\begin{equation}
\begin{aligned}
  - \frac{\tau_u}{h_u} c_1^2 \bip{e_E^T \bfu_t}{e_E^T \bfu - I_{v2u}^g e_W^T \bfv}{U} = &- \frac{\tau_u}{h_u} c_1^2 \bip{e_E^T \bfu_t}{e_E^T \bfu}{U} \\
  &+ \frac{\tau_u}{h_u} c_1^2 \bip{e_E^T \bfu_t}{I_{v2u}^g e_W^T \bfv}{U}. 
\end{aligned}
\end{equation}
We also write the second boundary integral in \eqref{ut_utt} as the sum of two integrals, and use the adjoint property of the interpolation operators in the first of them to obtain
\begin{equation}
\begin{aligned}
  -\frac{\sigma_u}{h_v} c_2^2 \bip{e_E^T \bfu_t}{I_{v2u}^b (I_{u2v}^g e_E^T \bfu - e_W^T \bfv)}{U} = &- \frac{\sigma_u}{h_v} c_2^2 \bip{I_{u2v}^g e_E^T \bfu_t}{I_{u2v}^g e_E^T \bfu}{V} \\ 
  &+ \frac{\sigma_u}{h_v} c_2^2 \bip{e_E^T \bfu_t}{ I_{v2u}^b e_W^T \bfv}{U}.
\end{aligned}
\end{equation}
For the third boundary integral in \eqref{ut_utt}, we have
\begin{equation}
  \frac{c_1^2}{2} \bip{d_E^T \bfu_t}{e_E^T \bfu - I_{v2u}^g e_W^T \bfv}{U} = \frac{c_1^2}{2} \bip{d_E^T \bfu_t}{e_E^T \bfu}{U} - \frac{c_1^2}{2} \bip{d_E^T \bfu_t}{I_{v2u}^g e_W^T \bfv}{U} .
\end{equation}
Using the adjoint property in the last boundary integral in \eqref{ut_utt} leads to
\begin{equation}
\begin{aligned}
  - \frac{1}{2} \bip{e_E^T \bfu_t}{c_1^2 d_E^T \bfu + c_2^2 I_{v2u}^b d_W^T \bfv}{U} = &- \frac{c_1^2}{2} \bip{e_E^T \bfu_t}{ d_E^T \bfu}{U} \\
   &- \frac{c_2^2}{2} \bip{I_{u2v}^{g} e_E^T \bfu_t}{d_W^T \bfv}{V} .
\end{aligned}
\end{equation}
Gathering terms, we obtain
\begin{equation} \label{eq:ut_utt_2}
\begin{aligned}
  \ip{\bfu_t}{\bfu_{tt}}{U} = &-\frac{c_1^2}{2}\dd{}{t} \left[ \frac{\tau_u}{h_u} \bnorm{e_E^T \bfu}{U}^2 - \bip{e_E^T \bfu}{d_E^T \bfu}{U} + \norm{\widetilde D_{\nabla} \bv{u}}_{U}^2 \right] \\
  & -\frac{\sigma_u}{h_v} \frac{c_2^2}{2} \dd{}{t} \bnorm{I_{u2v}^g e_E^T \bfu}{V}^2  \\
  &-\frac{c_1^2}{2} \bip{d_E^T \bfu_t}{I_{v2u}^g e_W^T \bfv}{U} + \frac{\tau_u}{h_u} c_1^2 \bip{e_E^T \bfu_t}{I_{v2u}^g e_W^T \bfv}{U}\\
  &+ \frac{\sigma_u}{h_v} c_2^2 \bip{I_{u2v}^g e_E^T \bfu_t}{e_W^T \bfv}{V} - \frac{c_2^2}{2} \bip{I_{u2v}^{g} e_E^T \bfu_t}{d_W^T \bfv}{V} .
\end{aligned}
\end{equation}
By repeating the procedure above for the second equation in \eqref{eq:scheme_wave}, we arrive at a similar expression for $\ip{\bfv_t}{\bfv_{tt}}{V}$: 
\begin{equation} \label{eq:vt_vtt}
\begin{aligned}
  \ip{\bfv_t}{\bfv_{tt}}{V} = &-\frac{c_2^2}{2}\dd{}{t} \left[ \frac{\tau_v}{h_v}  \bnorm{e_W^T \bfv}{V}^2 -  \bip{e_W^T \bfv}{d_W^T \bfv}{V} + \norm{\widetilde D_{\nabla} \bv{v}}_{V}^2 \right] \\ 
  & -\frac{\sigma_v}{h_u} \frac{c_1^2}{2} \dd{}{t} \bnorm{I_{v2u}^g e_W^T \bfv}{U}^2 \\
  &-\frac{c_2^2}{2} \bip{d_W^T \bfv_t}{I_{u2v}^g e_E^T \bfu}{V} + \frac{\tau_v}{h_v} c_2^2 \bip{e_W^T \bfv_t}{I_{u2v}^g e_E^T \bfu}{V}\\
  &+ \frac{\sigma_v}{h_u} c_1^2 \bip{I_{v2u}^g e_W^T \bfv_t}{ e_E^T \bfu}{U} - \frac{c_1^2}{2} \bip{I_{v2u}^{g} e_W^T \bfv_t}{d_E^T \bfu}{U} . \\ 
\end{aligned}
\end{equation}
Adding \eqref{eq:ut_utt_2} and \eqref{eq:vt_vtt} leads to
\begin{equation}
\begin{aligned}
  &\frac{1}{2} \dd{}{t} \left( \norm{\bfu_t}_U^2 + \norm{\bfv_t}_V^2 + c_1^2 \norm{\widetilde D_{\nabla} \bv{u}}_{U}^2 + c_2^2\norm{\widetilde D_{\nabla} \bv{v}}_{V}^2 \right) = \\
&- \frac{c_1^2}{2} \dd{}{t} \left[ \frac{\tau_u}{h_u} \bnorm{e_E^T \bfu}{U}^2 - \bip{e_E^T \bfu}{d_E^T \bfu}{U} + \frac{\sigma_v}{h_u} \bnorm{I_{v2u}^g e_W^T \bfv}{U}^2 \right] \\
&- \frac{c_2^2}{2}\dd{}{t} \left[ \frac{\tau_v}{h_v}  \bnorm{e_W^T \bfv}{V}^2 - \bip{e_W^T \bfv}{d_W^T \bfv}{V} + \frac{\sigma_u}{h_v} \bnorm{I_{u2v}^g e_E^T \bfu}{V}^2 \right] \\
& - \frac{c_1^2}{2} \bip{d_E^T \bfu_t}{I_{v2u}^g e_W^T \bfv}{U} - \frac{c_1^2}{2} \bip{I_{v2u}^{g} e_W^T \bfv_t}{d_E^T \bfu}{U} \\
&+ \frac{\tau_u}{h_u} c_1^2 \bip{e_E^T \bfu_t}{I_{v2u}^g e_W^T \bfv}{U} + \frac{\sigma_v}{h_u} c_1^2 \bip{I_{v2u}^g e_W^T \bfv_t}{ e_E^T \bfu}{U} \\
&+ \frac{\sigma_u}{h_v} c_2^2 \bip{I_{u2v}^g e_E^T \bfu_t}{e_W^T \bfv}{V} + \frac{\tau_v}{h_v} c_2^2 \bip{e_W^T \bfv_t}{I_{u2v}^g e_E^T \bfu}{V} \\
&+ \frac{c_2^2}{2} \bip{I_{u2v}^{g} e_E^T \bfu_t}{d_W^T \bfv}{V} -\frac{c_2^2}{2} \bip{d_W^T \bfv_t}{I_{u2v}^g e_E^T \bfu}{V} .
\end{aligned}
\end{equation}
The choice $\tau_u = \sigma_v$, $\tau_v = \sigma_u$ yields
\begin{equation} \label{eq:ddt_E_zero_discr_wave}
  \dd{}{t} \mathcal{E}_{wave} = 0,
 \end{equation} 
 where we have defined
 \begin{equation} \label{eq:discr_energy_wave}
 \begin{aligned}
  \mathcal{E}_{wave} &= \frac{1}{2} \left( \norm{\bfu_t}_U^2 + \norm{\bfv_t}_V^2 + c_1^2 \norm{\widetilde D_{\nabla} \bv{u}}_{U}^2  + c_2^2 \norm{\widetilde D_{\nabla} \bv{v}}_{V}^2 + c_1^2 A_u + c_2^2 A_v\right) ,
 \end{aligned}  
 \end{equation}
 with
\begin{equation}
\begin{aligned}
    A_u = \; & \frac{\tau_u}{h_u} \left(\bnorm{e_E^T \bfu}{U}^2 - 2 \bip{e_E^T \bfu}{I_{v2u}^g e_W^T \bfv}{U} + \bnorm{I_{v2u}^g e_W^T \bfv}{U}^2 \right) \\ 
    &- \bip{e_E^T \bfu}{d_E^T \bfu}{U} + \bip{d_E^T \bfu}{I_{v2u}^g e_W^T \bfv}{U} \\
    = \;& \frac{\tau_u}{h_u} \bnorm{e_E^T \bfu-I_{v2u}^g e_W^T \bfv}{U}^2 - \bip{d_E^T \bfu}{e_E^T \bfu - I_{v2u}^g e_W^T \bfv}{U},
\end{aligned}
\end{equation}
and
\begin{equation}
\begin{aligned}
  A_v = \frac{\tau_v}{h_v}  \bnorm{e_W^T \bfv - I_{u2v}^g e_E^T \bfu}{V}^2 - \bip{d_W^T \bfv}{e_W^T \bfv - I_{u2v}^{g} e_E^T \bfu}{V} .
\end{aligned}
\end{equation}
We note that $A_u$ and $A_v$ are zero to the order of accuracy because $I_{v2u}^g e_W^T \bfv \simeq e_E^T \bfu$ and $I_{u2v}^g e_E^T \bfu \simeq e_W^T \bfv$. Hence, the discrete energy $\mathcal{E}_{wave}$ mimics the continuous energy $E_{wave}$ in \eqref{eq:energy_cont_wave}.

For stability it remains to prove that we can choose $\tau_{u,v}$ so that $\mathcal{E}_{wave}$ is non-negative. By completing the squares in $A_{u,v}$ we obtain
\begin{equation}
  A_u = \frac{\tau_u}{h_u} \bnorm{e_E^T \bfu-I_{v2u}^g e_W^T \bfv - \frac{h_u}{2 \tau_u} d_E^T \bfu }{U}^2 - \frac{h_u}{4 \tau_u } \bnorm{d_E^T \bfu}{U}^2 ,
\end{equation}
and
\begin{equation}
  A_v = \frac{\tau_v}{h_v} \bnorm{e_W^T \bfv - I_{u2v}^g e_E^T \bfu - \frac{h_v}{2 \tau_v} d_W^T \bfv }{V}^2 - \frac{h_v}{4 \tau_v} \bnorm{d_W^T \bfv}{V}^2 .
\end{equation}
Because of \eqref{eq:borrowing_u} and \eqref{eq:borrowing_v}, we have
\begin{equation}
  \norm{\widetilde D_{\nabla} \bv{u}}_{U}^2 \geq h_u \gamma_u \bnorm{d_E^T \bv{u}}{U}^2, \quad \norm{\widetilde D_{\nabla} \bv{v}}_{V}^2 \geq h_v \gamma_v \bnorm{d_E^T \bv{v}}{V}^2.
\end{equation}
We can now derive a lower bound for $\mathcal{E}_{wave}$:
\begin{equation}
\begin{aligned}
  2\mathcal{E}_{wave} & = \norm{\bfu_t}_U^2 + \norm{\bfv_t}_V^2 + c_1^2 \norm{\widetilde D_{\nabla} \bv{u}}_{U}^2  + c_2^2 \norm{\widetilde D_{\nabla} \bv{v}}_{V}^2 + c_1^2 A_u + c_2^2 A_v \\
  & \geq  c_1^2 \left( \norm{\widetilde D_{\nabla} \bv{u}}_{U}^2 + A_u \right)  + c_2^2 \left( \norm{\widetilde D_{\nabla} \bv{v}}_{V}^2 + A_v \right)  \\
  & \geq  c_1^2 \left( h_u \gamma_u - \frac{h_u}{4 \tau_u } \right) \bnorm{d_E^T \bv{u}}{U}^2 + c_2^2 \left( h_v \gamma_v - \frac{h_v}{4 \tau_v}  \right) \bnorm{d_E^T \bv{v}}{V}^2 .
\end{aligned}
\end{equation}
It follows that $\mathcal{E}_{wave} \geq 0$ if
\begin{equation}
  \tau_u \geq \frac{1}{4 \gamma_u} , \quad \tau_v \geq \frac{1}{4 \gamma_v}.
 \end{equation} 
Hence, with appropriate values of the penalty parameters, the discrete energy $\mathcal{E}_{wave}$ is a semi-norm of $[\bfu, \bfv]$. The estimate \eqref{eq:ddt_E_zero_discr_wave} shows that $\mathcal{E}_{wave}$ is non-increasing in time. Because $\mathcal{E}_{wave}$ contains $\norm{\bfu_t}_U^2 + \norm{\bfv_t}_V^2$, one can show that the solution grows at most linearly in time, see e.g.\ \cite{Wang2017}. Thus, the scheme \eqref{eq:scheme_wave} is stable.
\end{proof}
\begin{remark}
Energy stability is guaranteed by Theorem \ref{thm:wave} as long as the penalty parameters satisfy $\theta_{u,v}\geq 1$. However, the choice $\theta_{u,v} = 1$ reduces the rank of the discrete spatial operator by 1, and leads to a convergence rate lower than the ideal rate \cite{Wang2017}. It is therefore important to set $\theta_{u,v} > 1$. Large values of  $\theta_{u,v}$ may improve the accuracy, but lead to a large spectral radius of the discretization matrix and hence a small time step in explicit time stepping methods \cite{MattssonHam09}. 
\end{remark}

\section{Three different discretizations of the Laplacian}\label{sec_Laplacian}
In this section we discuss the properties of the different discrete approximations of the Laplacian that have been introduced in this paper. Recall that, with Dirichlet or Neumann boundary conditions, the continuous Laplacian is symmetric and negative semidefinite in the $L^2$ inner product. We shall investigate the symmetry and definiteness of the discrete Laplacians.

To suppress unnecessary notation, we assume in this section that the PDE coefficients are continuous across the domain interface $\Gamma$ (the discussion applies to discontinuous coefficients too). That is, we consider \eqref{eq:SE_heat_2D_continuous} with $a=b=1$ (the heat equation) or $a=b=i$ (the Schrödinger equation), or the wave equation \eqref{eq:wave_2D_continuous} with $c_1=c_2=1$. 
Let $\bv{w}^T = \begin{bmatrix} \bv{u}^T & \bv{v}^T \end{bmatrix}$ and let $H$ denote the global quadrature defined by
\begin{equation}
  H = \begin{bmatrix} H_U & \\ & H_V \end{bmatrix} .
\end{equation}
The scheme \eqref{eq:scheme_schr_heat} can be written in the forms
\begin{equation}
  \bv{w}_t = D_{h} \bv{w} \, \mbox{(heat equation)} , \quad \quad \bv{w}_t = i D_{s} \bv{w} \, \mbox{(Schrödinger equation)} ,
\end{equation}
where $D_{h}$ and $D_{s}$ are different approximations of the Laplacian, including the SATs for the interface conditions. For stability, the symmetric parts of $HD_{h}$ and $iHD_{s}$  must be negative semidefinite. The penalty parameters in \eqref{eq:parameter_values} are such that $HD_{h}$ is non-symmetric but has a negative semidefinite symmetric part, while $HD_{s}$ is symmetric but (in all the cases we have investigated) indefinite. 

Similarly, the scheme \eqref{eq:scheme_wave} can be written as
\begin{equation} \label{eq:laplacian_wave}
  \bv{w}_{tt} = D_{w} \bv{w} \, \mbox{(wave equation)},
\end{equation}
where $D_w$ is the third type of approximation of the Laplacian, including the SATs for the interface conditions. 
The scheme \eqref{eq:laplacian_wave} is stable if $D_{w}$ is symmetric and negative semidefinite in the discrete inner product. The penalty parameters derived in Theorem \ref{thm:wave} ensure precisely this. Hence, $D_w$ mimics both of these properties of the continuous Laplacian. Note also that neither of the schemes
\begin{equation}
  \bv{w}_{tt} = D_{h} \bv{w} , \quad \bv{w}_{tt} = D_{s} \bv{w} ,
\end{equation}
is a stable discretization of the wave equation, but
\begin{equation}
    \bv{w}_t = D_{w} \bv{w} \quad \mbox{and} \quad \bv{w}_t = iD_{w} \bv{w}
 \end{equation} 
are stable discretizations of the heat and Schrödinger equations. So, for the heat and Schrödinger equations, we have two possible discrete Laplacians. While $D_h$ and $D_s$ are simple in the sense that they involve fewer penalty terms, one might argue that $D_w$ could be a better discretization of the Laplacian since it too is symmetric and negative semidefinite. In section \ref{sec:num_exp}, we show that it can be beneficial to use $D_w$ in place of $D_s$ when discretizing the Schrödinger equation, because it leads to smaller errors and smoother convergence behaviour. 

\section{Numerical experiments} \label{sec:num_exp}
In this section we present numerical experiments with the heat, Schrödinger, and wave equations. We use narrow-stencil diagonal-norm SBP operators \cite{MattssonNordstrom04} of interior orders $2p=4$ and $2p=6$ to approximate the spatial derivatives. We compare the new interface treatment based on OP interpolation operators with previous approaches that use only one pair of interpolation operators and hence suffer from accuracy reduction. For the non-OP schemes, we use interpolation operators developed by Mattsson and Carpenter \cite{MattssonCarpenter09}. The new and old schemes are abbreviated as OP and MC, respectively. We also let $\R^2_L$ and $\R^2_R$ denote the left and right half planes, respectively.

The reason for not including second order accurate schemes in the comparison is that MC discretizations converge with the ideal second order rate and there is nothing to gain in using OP interpolation.

\subsection{The heat equation}
We consider the heat equation 
\begin{equation} \label{heat_eqn}
\begin{array}{rcl}
u_t - \lambda_1 \Delta u = 0, & (x,y) \in \R_L^2, & t>0 , \\
v_t - \lambda_2 \Delta v = 0, & (x,y) \in \R_R^2, & t>0 , \\
u - v = 0, & x = 0 , & t > 0 , \\
\lambda_1 u_x - \lambda_2 v_x = 0, & x = 0 , & t > 0 ,
\end{array}
\end{equation}
where the diffusion coefficients $\lambda_1$ and $\lambda_2$ are constant.
\begin{figure}
\centering
\includegraphics[trim={0.9cm 2.5cm 1.2cm 3.5cm}, clip, width=0.8 \textwidth]{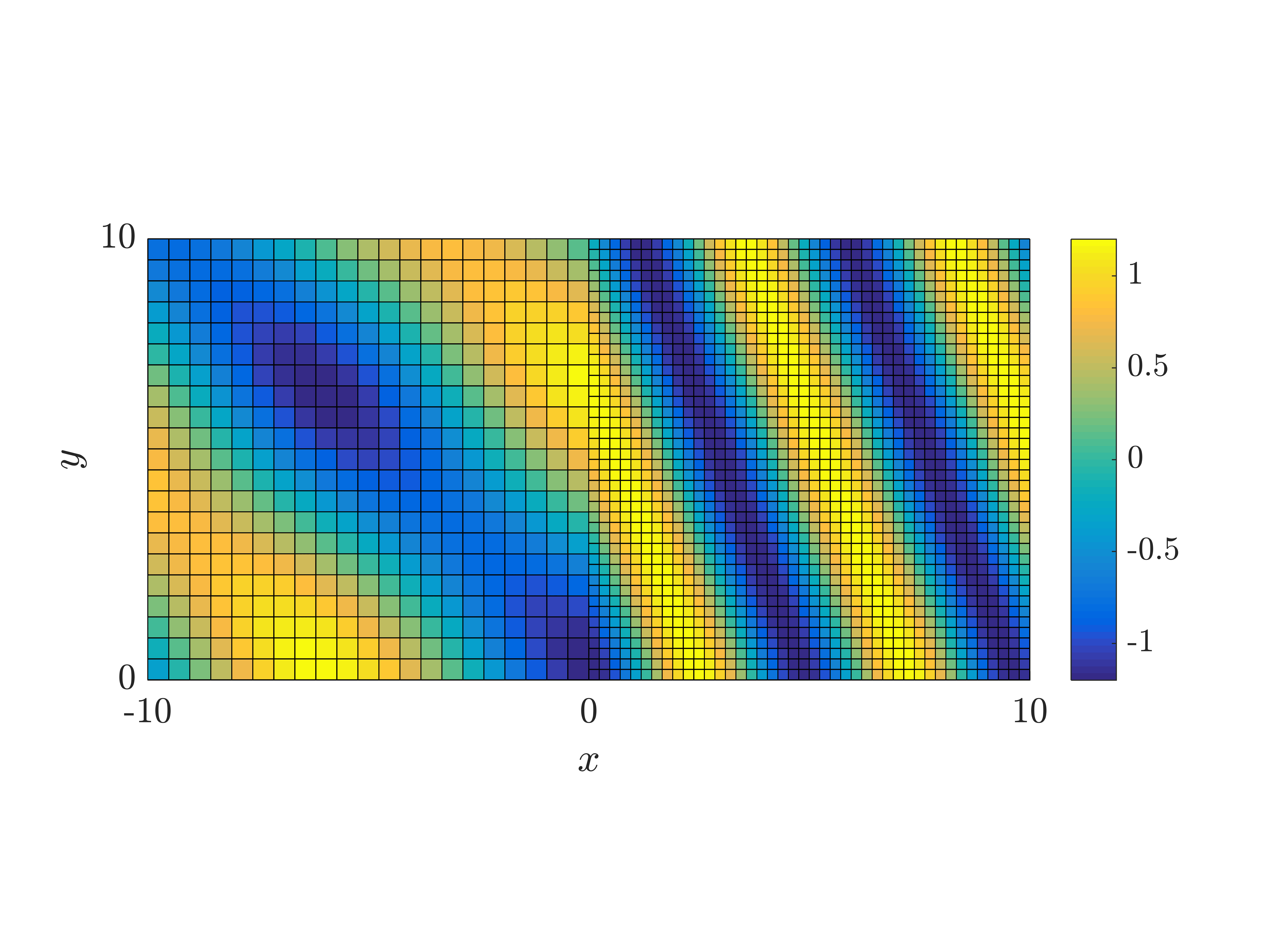}
\caption{The exact solution of the heat equation at time $t=0$, plotted on a grid of $21\times 21$ points in the left block and $41\times 41$ points in the right block.}
\label{heat_solution}
\end{figure}
The equation admits analytical solutions in the form
\begin{equation} 
\begin{aligned}
u &= \cos(k_1x+k_2y)e^{-\omega t} +  \gamma\cos(k_1x-k_2y)e^{-\omega t}, \\
v &= (1+\gamma)\cos(kx+k_2y)e^{-\omega t},
\end{aligned}
\label{heat_analytical_solution}
\end{equation}
where $\omega=\lambda_1(k_1^2+k_2^2)$, $k=\sqrt{\omega/\lambda_2-k_2^2}$ and $\gamma=(\lambda_1 k_1-\lambda_2 k)/(\lambda_1 k_1+\lambda_2 k)$. 
We choose the diffusion coefficients $\lambda_1=0.1$ and $\lambda_2=0.025$ and set $k_1=k_2=0.5$. The exact solution corresponding to these parameter values at the initial time $t=0$ is plotted in Figure \ref{heat_solution}. In the computation we restrict the domain to $[-10, 10] \times [0, 10]$, impose Dirichlet boundary conditions at all outer boundaries, and use the exact solution to obtain initial and boundary data. The Dirichlet boundary conditions are imposed weakly by the SAT method \cite{Berg2012heat}.

Because the diffusion coefficient ratio is $\lambda_1/\lambda_2=4$, the spatial frequency in the right half plane is twice as large as that in the left half plane. To resolve this solution efficiently, we use Cartesian grids with grid sizes $h_u$ and $h_v=0.5h_u$ in the left and right blocks, respectively. This results in a non-conforming interface at $x=0$ with mesh refinement ratio 1:2, as shown in Figure \ref{heat_solution}.  Equation \eqref{heat_eqn} is discretized in space by the scheme \eqref{eq:scheme_schr_heat}, and is integrated in time by the 4th order backward differentiation formula with a time step $\Delta t=0.25h_v$. We set the final time $T=2$, at which point the exact solution has the same shape as the initial solution, with the maximum amplitude diffused from 1.20 to 1.09.

\begin{figure}
\centering
\includegraphics[width=0.7\textwidth]{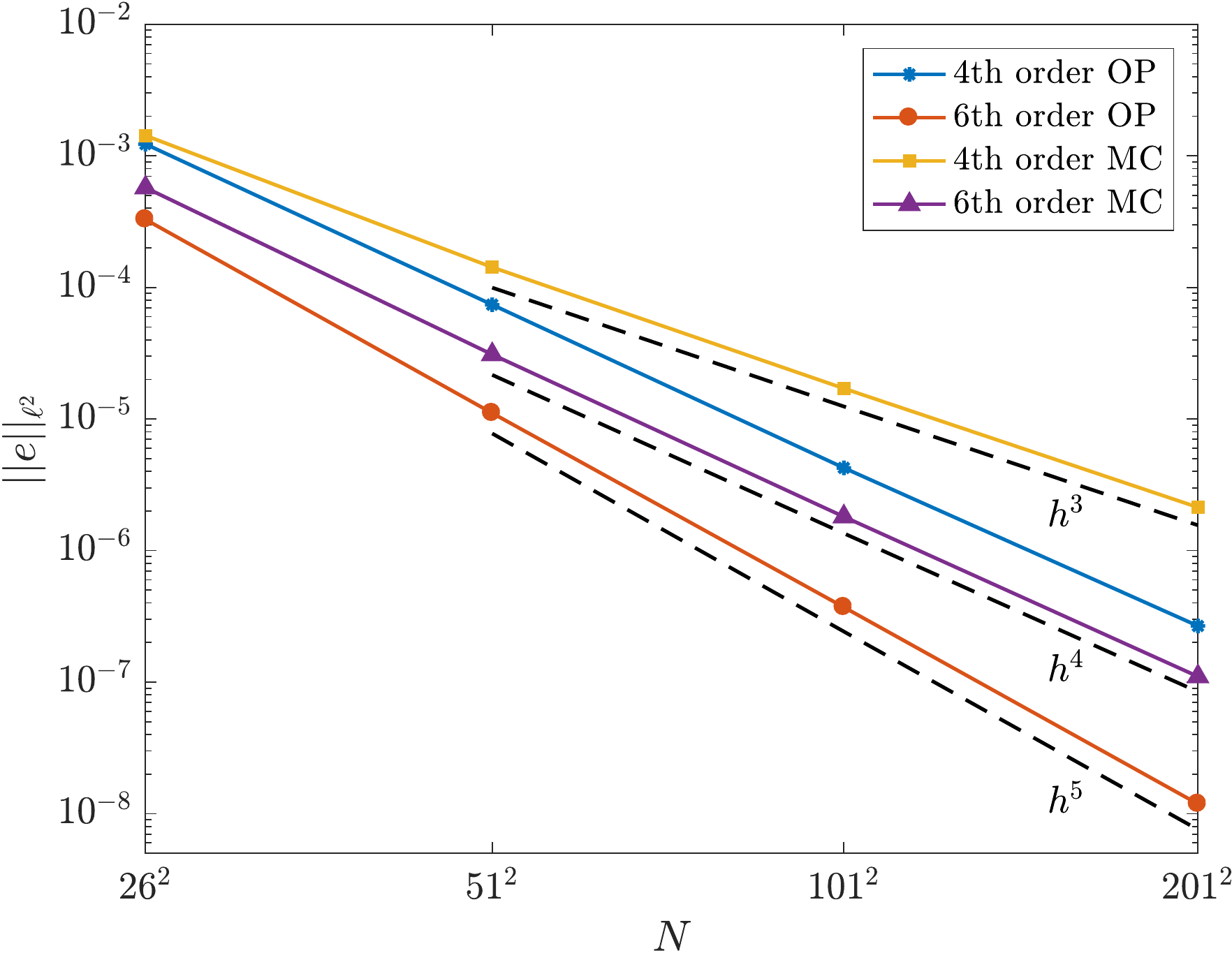}
\caption{Error plot for the heat equation. OP vs.\ MC interpolation operators. $N$ denotes the number of grid points in the coarse block.}
\label{heat_convergence}
\end{figure}

In Figure \ref{heat_convergence} we plot the $\ell^2$-errors of the solution at time $t=2$. We observe that with the MC interpolation operators, the convergence rate is $p+1$, and with the OP interpolation operators, the rate is $p+2$, where $p = 2,3$.

We have also performed the above experiment with the symmetric, negative semidefinite spatial discretization in \eqref{eq:scheme_wave}, and obtained similar $\ell^2$-errors and convergence rates. This suggests that the slightly simpler scheme \eqref{eq:scheme_schr_heat} works well for the heat equation, if solution accuracy is the primary concern. However, the symmetric discrete operator is self-adjoint just like the continuous spatial operator, and hence \eqref{eq:scheme_wave} is a dual-consistent scheme, while \eqref{eq:scheme_schr_heat} is not \cite{Eriksson2017}.

\subsection{The Schr\"{o}dinger equation}
Consider the time dependent Schr\"{o}dinger equation with a potential step,
\begin{equation} \label{eq:SE_num_exp_setup}
\begin{array}{lcl}
u_t = i \Delta u  , & (x,y) \in \R^2_{L} , & t>0 , \\
v_t = i \Delta v + iV_0 v , & (x,y) \in \R^2_R , & t>0 , \\
u = v , & x = 0 , & t > 0 , \\
u_x = v_x , & x = 0 , & t > 0 ,
\end{array}
\end{equation}
where $V_0$ is constant. The equation \eqref{eq:SE_num_exp_setup} admits exact solutions of the form
\begin{equation} \label{eq:analytical_Schrodinger}
  \begin{aligned} 
  u(x,y,t) &= A e^{i(k_1 x + k_2 y - \omega t)} + B e^{i(-k_1 x + k_2 y - \omega t)},  \\
  v(x,y,t) &= C e^{i\left( \widetilde{k}_1 x + k_2 y - \omega t \right)},
   \end{aligned}
\end{equation}
where
\begin{equation}
  \omega = k_1^2+k_2^2 , \quad \widetilde{k}_1 = \sqrt{V_0 + k_1^2} , \quad B = A \frac{k_1-\widetilde{k}_1}{k_1+\widetilde{k}_1} , \quad C = A + B .
\end{equation}
We set $A=1$, $V_0=3\pi^2$ and $k_1=k_2=\pi$, which yields $\omega = 2 \pi^2$, $\widetilde{k}_1 = 2\pi$, $B=-\frac{1}{3}$ and $C=\frac{2}{3}$. In the computations we restrict the spatial domain to $[-1,1] \times [0,1]$. We impose Dirichlet boundary conditions and use the exact solution as initial and boundary data. Because the solution has a larger wavenumber for $x>0$ we use two blocks with a 2:1 grid size ratio as depicted in Figure \ref{fig:schrodinger_initial}. We use the SBP in time method \cite{NORDSTROM2013} with an operator based on the Gauss quadrature rule with 4 points \cite{Boom_15} to advance the solution to the final time $T=0.5$. The time step is chosen as $\Delta t = 0.1h_{v}$. Numerical experiments indicate that this time step is small enough that the spatial errors dominate. We use 4th and 6th order spatial discretizations and compare MC with OP interpolation operators.

When using the indefinite discretization of the Laplacian, both MC and OP exhibit erratic convergence rates in the 6th order case, see Figures \ref{fig:schrodinger_mc_def_vs_indef} and \ref{fig:schrodinger_op_def_vs_indef}. Similar behavior was observed in \cite{Nissen_15}. When using the semidefinite discrete Laplacian instead, the convergence is smoother and the errors are smaller, in particular for 6th order. Hence, we propose to always use the semidefinite Laplacian for the Schrödinger equation, even though the indefinite Laplacian also is stable. 

Note that when switching to the semidefinite Laplacian, we changed not only the interface coupling but also the SATs that impose the Dirichlet boundary conditions on the outer boundaries. In the indefinite Laplacian we used the Dirichlet treatment in \cite{Nissen_13}, and in the semidefinite Laplacian we impose the Dirichlet conditions as in \cite{MattssonHam09}.  Replacing only the interface treatment or only the boundary treatment did not result in significantly improved convergence behavior. In the semidefinite Laplacian, we used the penalty strength 1.2 for both boundary and interface SATs.

Figure \ref{fig:schrodinger_op_vs_mc} compares OP with MC, when using the semidefinite Laplacian. As hypothesized, MC converges with rate $p+1$ while OP converges with rate $p+2$.

\begin{figure}[htbp]
\centering
\includegraphics[trim={1.2cm 2.5cm 1.2cm 3.5cm}, clip, width=0.8 \textwidth]{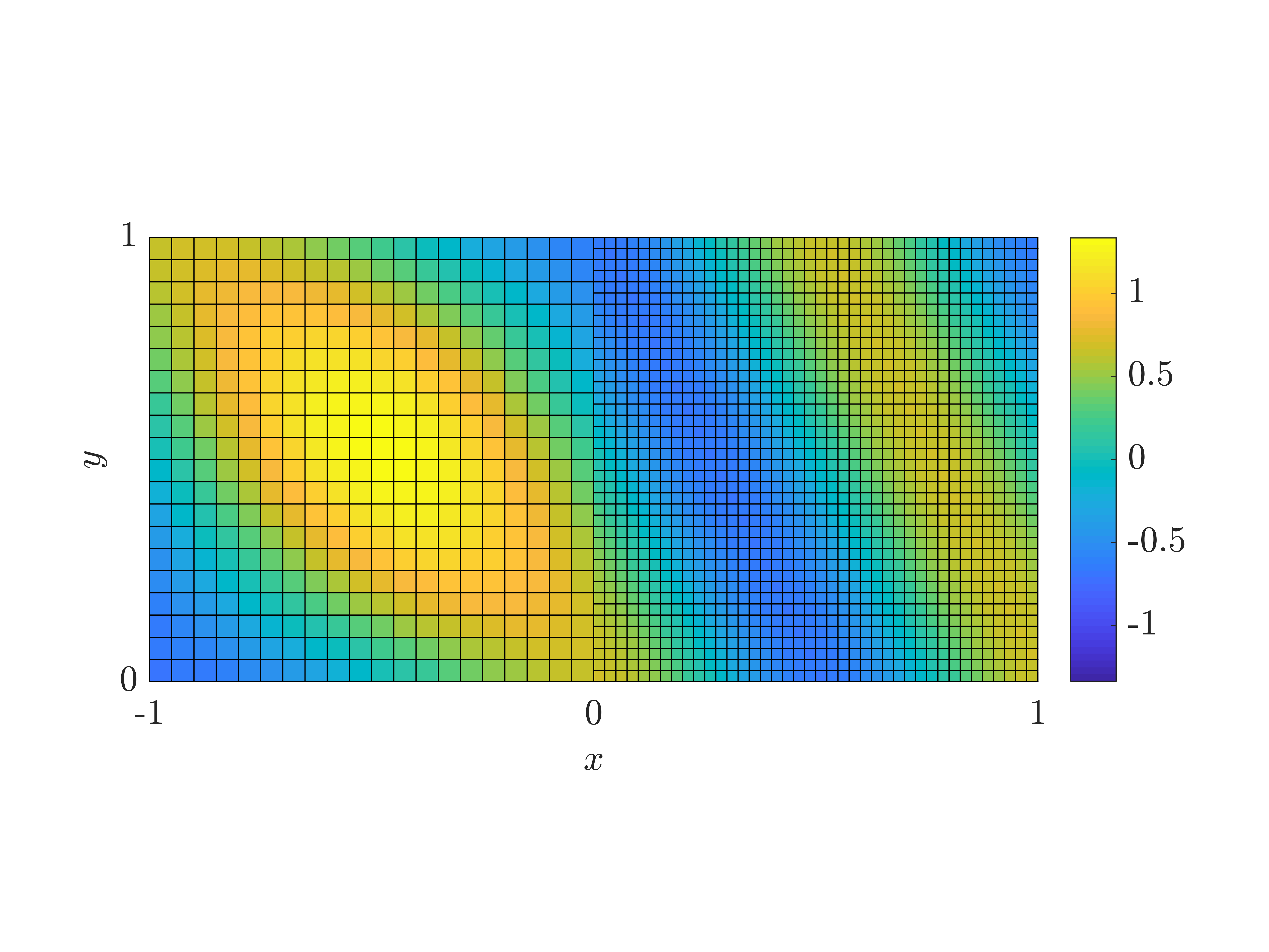}
\caption{The exact solution in the experiments with the Schrödinger equation at time $t=0$, plotted on a grid of $21\times 21$ points in the left block and $41\times41$ points in the right block.}
\label{fig:schrodinger_initial}
\end{figure}

\begin{figure}[htbp]
\centering
\includegraphics[width=0.7 \textwidth]{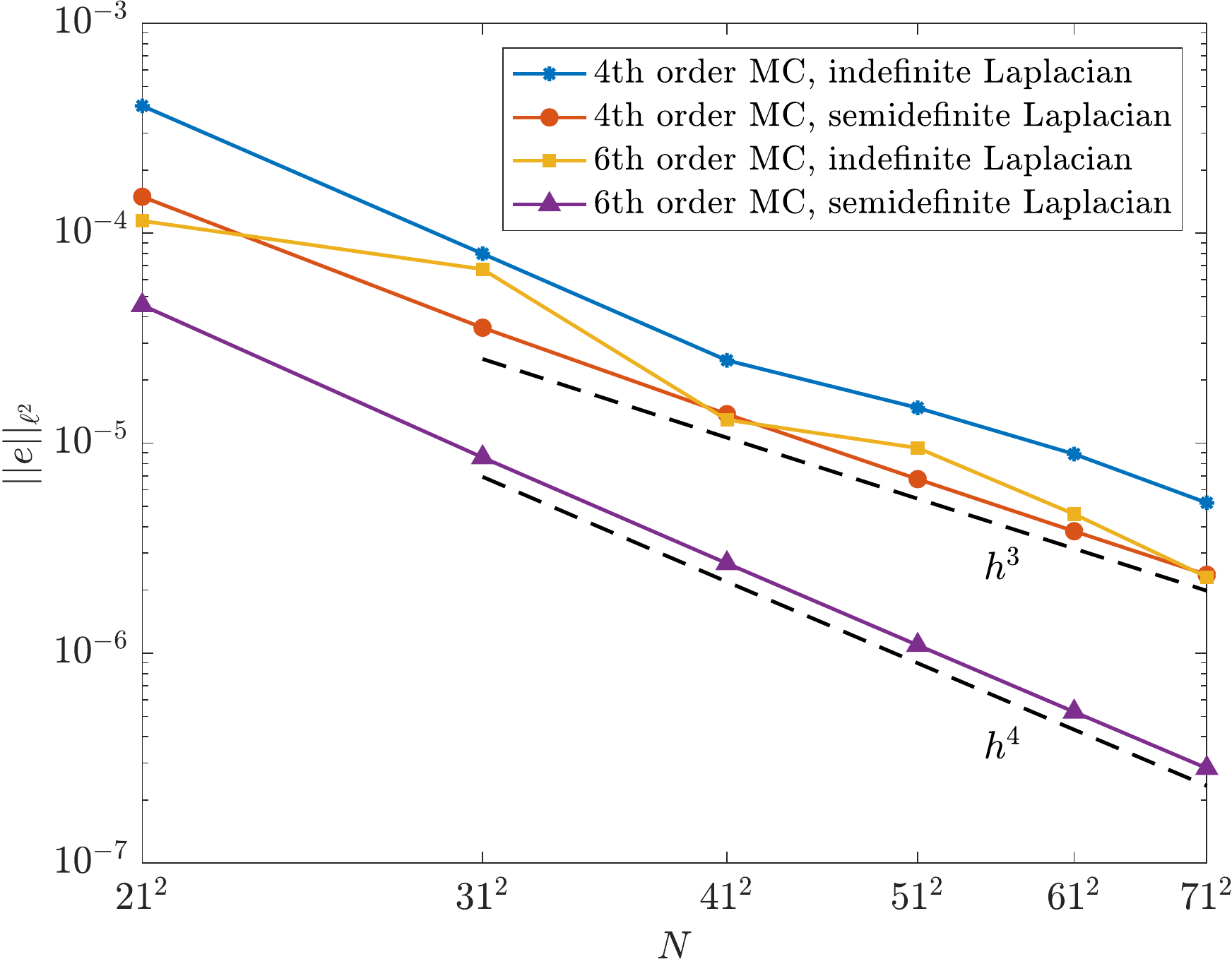}
\caption{Error plot for the Schrödinger equation using MC interpolation operators, comparing semidefinite and indefinite discretizations of the Laplacian. $N$ denotes the number of grid points in the coarse block.} 
\label{fig:schrodinger_mc_def_vs_indef}
\end{figure}

\begin{figure}[htbp]
\centering
\includegraphics[width=0.7 \textwidth]{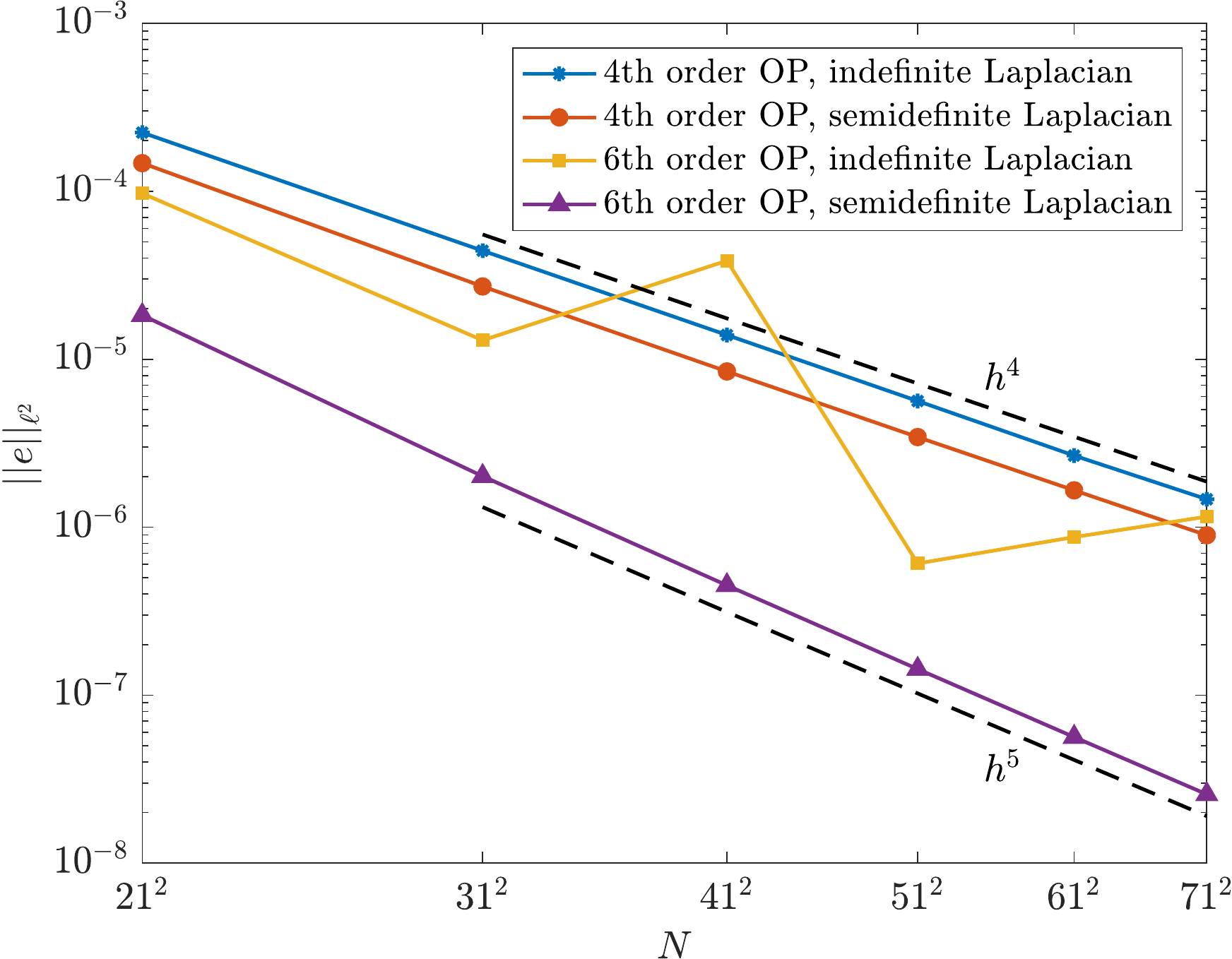}
\caption{Error plot for the Schrödinger equation using OP interpolation operators, comparing semidefinite and indefinite discretizations of the Laplacian. $N$ denotes the number of grid points in the coarse block.} 
\label{fig:schrodinger_op_def_vs_indef}
\end{figure}

\begin{figure}[htbp]
\centering
\includegraphics[width=0.7 \textwidth]{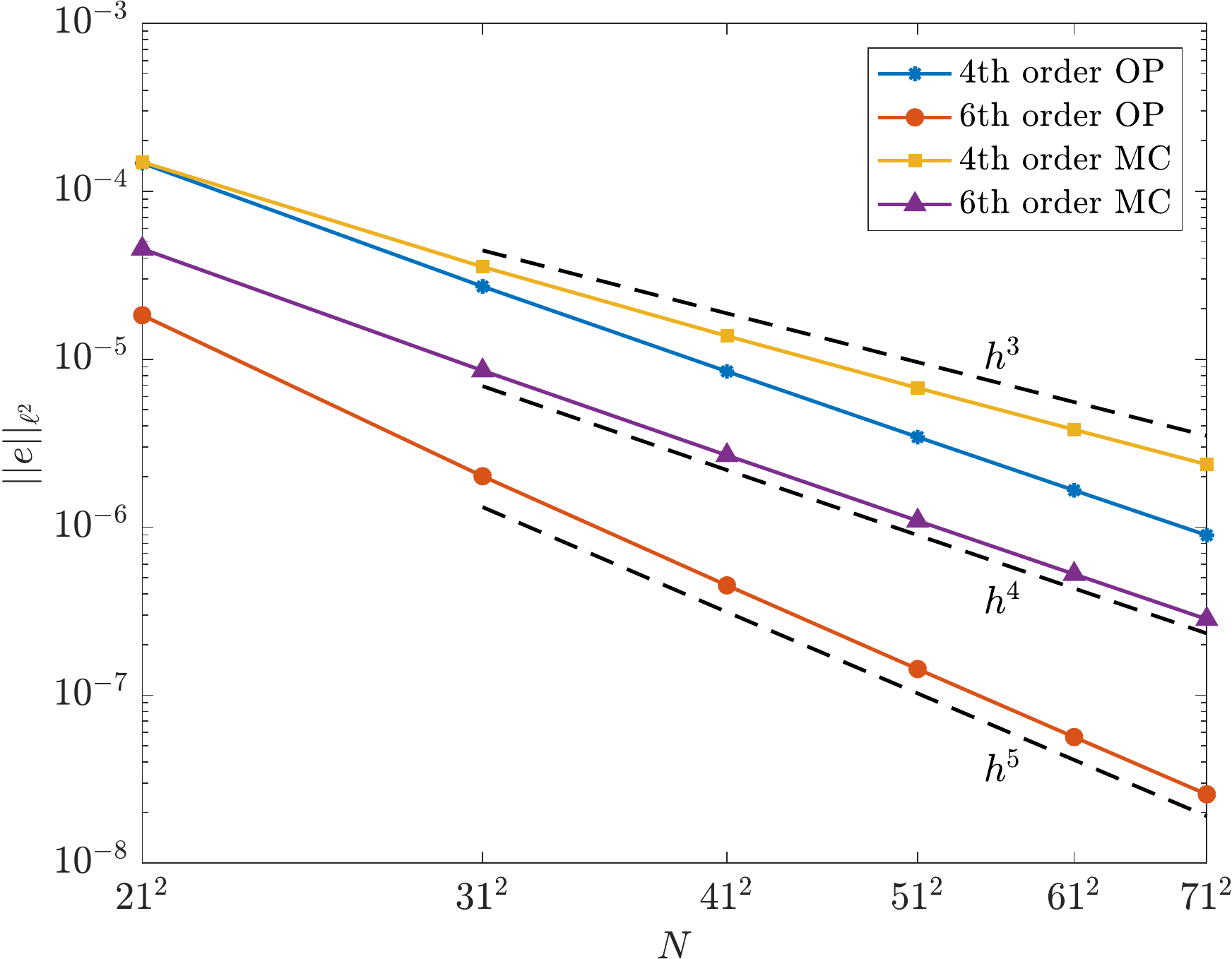}
\caption{Error plot for the Schrödinger equation. OP vs.\ MC interpolation operators, using the semidefinite discretization of the Laplacian. $N$ denotes the number of grid points in the coarse block.} 
\label{fig:schrodinger_op_vs_mc}
\end{figure}

\subsection{The wave equation}
We consider the wave equation 
\begin{equation} \label{wave_eqn}
\begin{array}{rcl}
u_{tt} - c_1^2 \Delta u = 0, & (x,y) \in \R_L^2, & t>0 , \\
v_{tt} - c_2^2 \Delta v = 0, & (x,y) \in \R_R^2, & t>0 , \\
u - v = 0, & x = 0 , & t > 0 , \\
c_1^2 u_x - c_2^2 v_x = 0, & x = 0 , & t > 0 ,
\end{array}
\end{equation}
By using Snell's law, we can derive an analytical solution in the form
\begin{equation}\label{wave_analytical}
\begin{aligned}
&u = \cos(x+y-\sqrt{2}c_1t)+k_2\cos(x-y+\sqrt{2}c_1t), \\
&v = (1+k_2)\cos(k_1 x+y+\sqrt{2}c_1t),
\end{aligned}
\end{equation}
where $k_1=\sqrt{2c_1^2/c_2^2-1}$ and $k_2=(c_1^2-c_2^2 k_1)/(c_1^2+c_2^2 k_1)$.  

In the experiment, we consider a piecewise constant wave speed by setting $c_1=1$ and $c_2=0.5$. This choice makes the wave number in the right half plane twice as large as that in the left half plane, which can be seen in the plot of the exact solution at time $t=0$ in Figure \ref{fig:wave_solution}. We restrict the domain to $[-10, 10] \times [0, 10]$, impose Dirichlet boundary conditions at all outer boundaries, and use the exact solution \eqref{wave_analytical} to obtain the initial and boundary data. To keep the number of grid points per wavelength approximately constant, we use a Cartesian mesh with mesh size $h_u$ in the left block, and $h_v=0.5h_u$ in the right block. 

\begin{figure}
\centering
\includegraphics[trim={0.9cm 2.5cm 1.2cm 3.5cm}, clip, width=0.8 \textwidth]{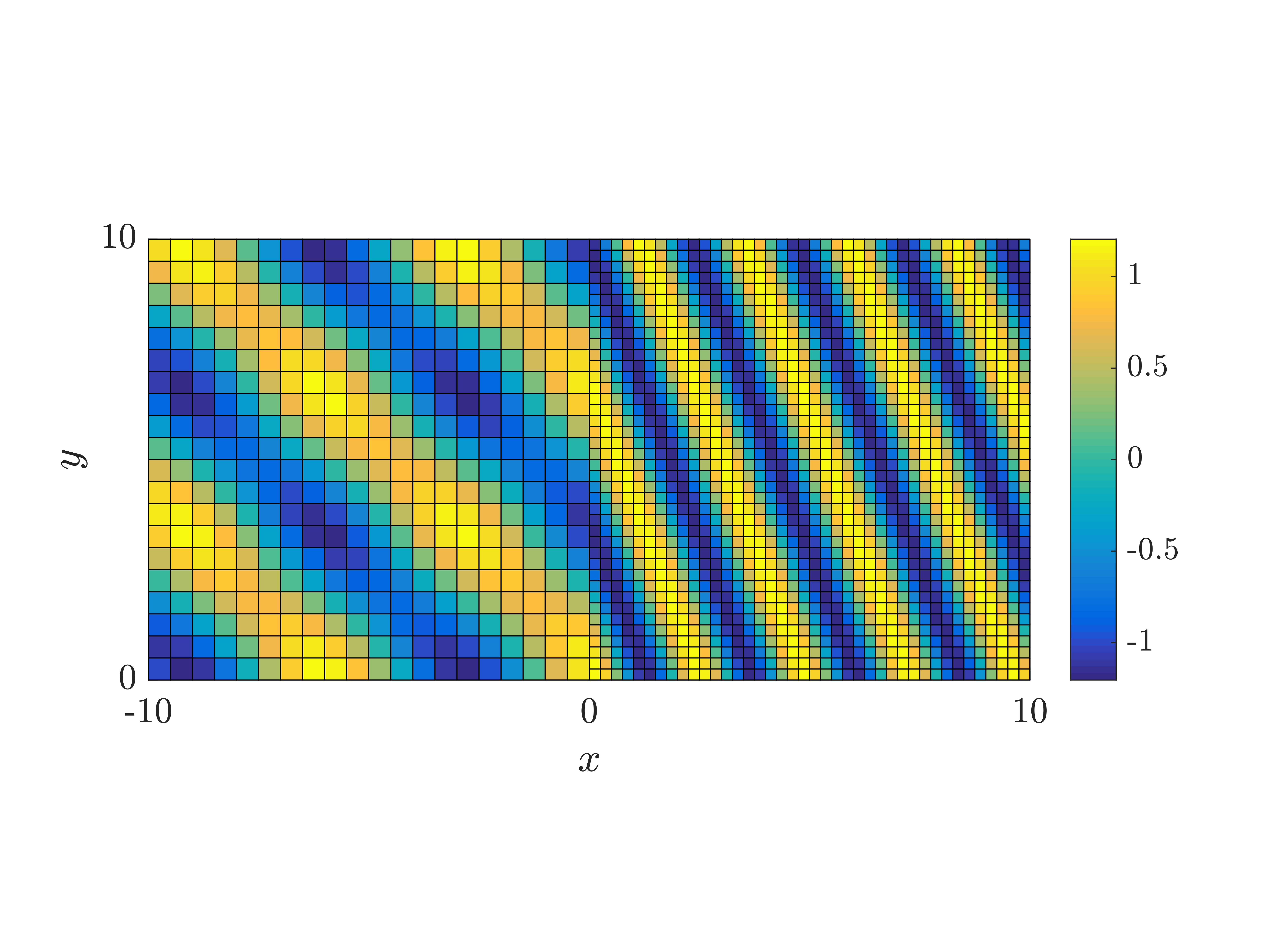}
\caption{The exact solution of the wave equation at time $t=0$, plotted on a grid of $21\times 21$ points in the left block and $41\times 41$ points in the right block.}
\label{fig:wave_solution}
\end{figure}

Equation \eqref{wave_eqn} is discretized in space by the scheme \eqref{eq:scheme_wave}, with either MC or OP interpolation operators. The Dirichlet boundary conditions are imposed weakly by the SAT method \cite{Appelo07,MattssonHam09}. We choose the values $\theta_u=\theta_v=3$ for the penalty parameters in Theorem \ref{thm:wave}, i.e.\ three times the limit value required for energy stability. In the penalty terms corresponding to the Dirichlet boundary conditions, we also set the penalty parameter to three times the limit value. We use the classical 4th order Runge--Kutta method to advance the solution to time $T=2$. The time step is chosen as $\Delta t=0.1h_v$, which is small enough that the error is dominated by the spatial discretization. In the error plot in Figure \ref{wave_convergence}, it is clear that the convergence rate for the scheme with the OP operators is $p+2$, whereas for MC it is $p+1$, where $p=2,3$.

\begin{figure}
\centering
\includegraphics[width=0.7\textwidth]{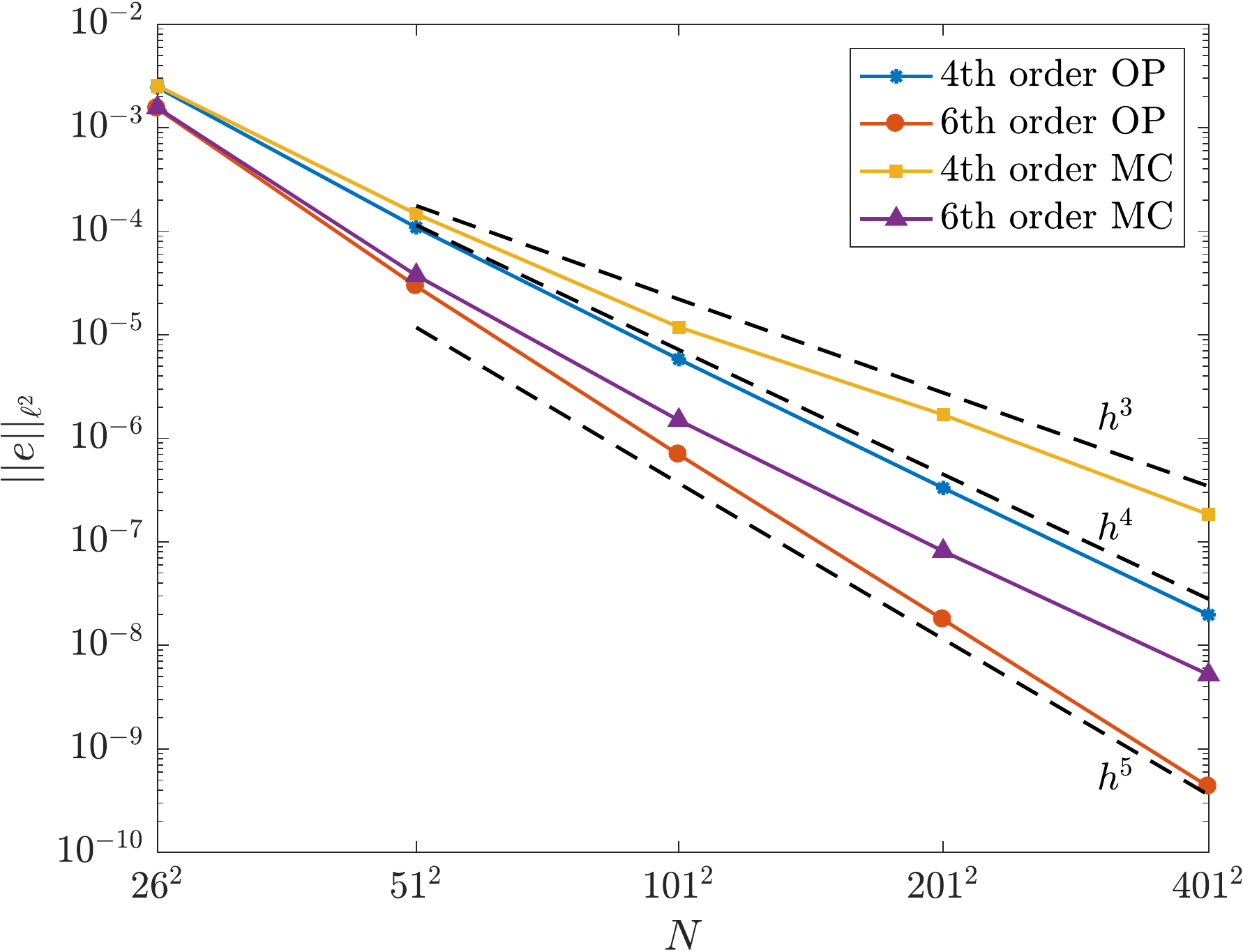}
\caption{Error plot for the wave equation. OP vs.\ MC interpolation operators. $N$ denotes the number of grid points in the coarse block.}
\label{wave_convergence}
\end{figure}

\section{Conclusion} \label{sec:conclusion}
We have studied non-conforming grid interfaces for time dependent partial differential equations with second derivatives in space. To remedy previously observed decreases in convergence rates, we have introduced order preserving (OP) interpolation operators for the non-conforming grid interfaces. The schemes based on OP operators are energy-stable and decrease the largest local truncation errors by one order compared to previous approaches. Numerical experiments demonstrate that the smaller truncation errors lead to an improvement of one order in global convergence rates.

The OP interpolation operators come in two pairs, where the two operators in a pair are the Hilbert adjoints of one another, i.e.\
\begin{equation}
  I_{v2u}^b = (I_{u2v}^g)^{\dagger} , \quad  I_{v2u}^g = (I_{u2v}^b)^{\dagger} ,
\end{equation}
where the inner products of the Hilbert spaces are defined by the quadrature rules that accompany the SBP operators. Let $q(I)$ denote the order of accuracy of the interpolation operator $I$. For traditional diagonal-norm SBP operators of interior order $2p$, it has previously been shown that $q(I) + q(I^{\dagger}) \leq 2p+1$. Theorem \ref{thm:existence} in this paper shows that given two quadrature rules of order $2p$, it is always possible to obtain $q(I) + q(I^{\dagger}) = 2p+1$. Moreover, the total order of $2p+1$ may be distributed arbitrarily between $I$ and $I^{\dagger}$. This guarantees that OP operators with the desired properties exist. 

To summarize, we propose the following schemes. For the second order wave equation, we propose the scheme \eqref{eq:scheme_wave}, with penalty parameters as in Theorem \ref{thm:wave}. For the Schrödinger equation, i.e.\ \eqref{eq:SE_heat_2D_continuous} with $a= i \alpha$ and $b = i \beta$, where $\alpha, \beta \in \R$, we propose to use the spatial operator resulting from the wave equation scheme \eqref{eq:scheme_wave}, with $c_1^2$ replaced by $i \alpha$ and $c_2^2$ replaced by $i \beta$. For the heat equation, i.e.\ \eqref{eq:SE_heat_2D_continuous} with $a,b \in \R$, one may use either the scheme \eqref{eq:scheme_schr_heat} or the spatial operator resulting from the wave equation scheme \eqref{eq:scheme_wave} with $c_1^2$ replaced by $a$ and $c_2^2$ replaced by $b$.

\section*{Acknowledgements}
M.\ Almquist gratefully acknowledges support from the Knut and Alice Wallenberg Foundation, Dnr 2016.0498.


\begin{thebibliography}{10}

\bibitem{Appelo07}
{\sc D.~Appel\"{o} and G.~Kreiss}, {\em Application of a perfectly matched
  layer to the nonlinear wave equation}, Wave Motion, 44 (2007), pp.~531--548.

\bibitem{Berg2012heat}
{\sc J.~Berg and J.~Nordstr\"{o}m}, {\em Spectral analysis of the continuous
  and discretized heat and advection equation on single and multiple domains},
  Appl. Numer. Math., 62 (2012), pp.~1620--1638.

\bibitem{Boom_15}
{\sc P.~D. Boom and D.~W. Zingg}, {\em High-order implicit time-marching
  methods based on generalized summation-by-parts operators}, SIAM J.
  Sci. Comput., 37 (2015), pp.~A2682--A2709.

\bibitem{CarpenterGottlieb94}
{\sc M.~H. Carpenter, D.~Gottlieb, and S.~Abarbanel}, {\em {T}ime-stable
  boundary conditions for finite-difference schemes solving hyperbolic systems:
  {M}ethodology and application to high-order compact schemes}, J. Comput.
  Phys., 111(2) (1994), pp.~220--236.

\bibitem{Eriksson2017}
{\sc S.~Eriksson}, {\em {A dual consistent finite difference method with narrow
  stencil second derivative operators}}, J. Sci. Comput., 75
  (2017), pp.~1--35.

\bibitem{DelReyFernandez2014a}
{\sc D.~C. {Del Rey Fern\'{a}ndez}, J.~E. Hicken, and D.~W. Zingg}, {\em Review of
  summation-by-parts operators with simultaneous approximation terms for the
  numerical solution of partial differential equations}, Comput. \& Fluids,
  95 (2014), pp.~171--196.

\bibitem{Friedrich2017}
{\sc L.~Friedrich, D.~C. {Del Rey Fern{\'{a}}ndez}, A.~R. Winters, G.~J.
  Gassner, D.~W. Zingg, and J.~Hicken}, {\em {Conservative and stable degree
  preserving SBP operators for non-conforming meshes}}, J. Sci.
  Comput., 75 (2017), pp.~1--30.

\bibitem{Gao2018}
{\sc L.~Gao and D.~Keyes}, {\em {Combining finite element and finite difference methods for isotropic elastic wave simulations in an energy-conserving
  manner}}, arXiv:1802.08324 [math.NA] (2018).

\bibitem{Hicken13}
{\sc J.~Hicken and D.~Zingg}, {\em {Summation-by-parts operators and high-order
  quadrature}}, J. Comput. Appl. Math., 237 (2013),
  pp.~111--125.

\bibitem{Kozdon2014}
{\sc J.~E. Kozdon and L.~C. Wilcox}, {\em {Stable coupling of nonconforming,
  high-order finite difference methods}}, SIAM J. Sci. Comput.,
  38 (2014), pp.~923--952.

\bibitem{KreissScherer74}
{\sc H.-O. Kreiss and G.~Scherer}, {\em Finite element and finite difference
  methods for hyperbolic partial differential equations.}, Mathematical Aspects
  of Finite Elements in Partial Differential Equations., Academic Press, Inc.,
  (1974).

\bibitem{Lundquist2018}
{\sc T.~Lundquist, A.~Malan, and J.~Nordstr{\"{o}}m}, {\em {A hybrid framework
  for coupling arbitrary summation-by-parts schemes on general meshes}},
  J. Comput. Phys., 362 (2018), pp.~49--68.

\bibitem{Lundquist2015}
{\sc T.~Lundquist and J.~Nordstr{\"{o}}m}, {\em {On the suboptimal accuracy of
  summation-by-parts schemes with non-conforming block interfaces}}, Tech. Report, LiTH-MAT-R--2015/16--SE (2015).

\bibitem{Mattsson11}
{\sc K.~Mattsson}, {\em Summation by parts operators for finite difference
  approximations of second-derivatives with variable coefficients}, J. Sci. Comput., 51 (2012), pp.~650--682.

\bibitem{MattssonCarpenter09}
{\sc K.~Mattsson and M.~H. Carpenter}, {\em Stable and accurate interpolation
  operators for high-order multi-block finite-difference methods}, SIAM J. Sci. Comput., 32(4) (2010), pp.~2298--2320.

\bibitem{MattssonHam08}
{\sc K.~Mattsson, F.~Ham, and G.~Iaccarino}, {\em Stable and accurate wave
  propagation in discontinuous media}, J. Comput. Phys., 227 (2008),
  pp.~8753--8767.

\bibitem{MattssonHam09}
{\sc K.~Mattsson, F.~Ham, and G.~Iaccarino}, {\em Stable boundary treatment for
  the wave equation on second-order form}, J. Sci. Comput., 41
  (2009), pp.~366--383.

\bibitem{MattssonNordstrom04}
{\sc K.~Mattsson and J.~Nordstr{\"o}m}, {\em Summation by parts operators for
  finite difference approximations of second derivatives}, J. Comput. Phys.,
  199(2) (2004), pp.~503--540.

\bibitem{Nissen_15}
{\sc A.~Nissen, K.~Kormann, M.~Grandin, and K.~Virta}, {\em Stable difference
  methods for block-oriented adaptive grids}, J. Sci. Comput.,
  65 (2015), pp.~486--511.

\bibitem{Nissen2012a}
{\sc A.~Nissen, G.~Kreiss, and M.~Gerritsen}, {\em {Stability at nonconforming
  grid interfaces for a high order discretization of the Schr{\"{o}}dinger
  equation}}, J. Sci. Comput., 53 (2012), pp.~528--551.

\bibitem{Nissen_13}
{\sc A.~Nissen, G.~Kreiss, and M.~Gerritsen}, {\em High order stable finite
  difference methods for the {S}chr\"odinger equation}, J. Sci.
  Comput., 55 (2013), pp.~173--199.

\bibitem{NORDSTROM2013}
{\sc J.~Nordström and T.~Lundquist}, {\em Summation-by-parts in time}, J. Comput. Phys., 251 (2013), pp.~487--499.

\bibitem{strand94}
{\sc B.~Strand}, {\em Summation by parts for finite difference approximations
  for d/dx}, J. Comput. Phys., 110 (1994), pp.~47--67.

\bibitem{SvardNordstrom06}
{\sc M.~Sv{\"a}rd and J.~Nordstr{\"o}m}, {\em On the order of accuracy for
  difference approximations of initial-boundary value problems}, J. Comput.
  Phys., 218 (2006), pp.~333--352.

\bibitem{Svard2014}
{\sc M.~Sv\"ard and J.~Nordstr\"om}, {\em Review of
  summation-by-parts-operators schemes for initial-boundary-value problems},
  J. Comput. Phys., 268 (2014), pp.~17--38.

\bibitem{Wang2017int}
{\sc S.~Wang}, {\em {An improved high order finite difference method for
  non-conforming grid interfaces for the wave equation}}, Accepted in J. Sci. Comput.,  (2018).

\bibitem{Wang2017}
{\sc S.~Wang and G.~Kreiss}, {\em Convergence of summation-by-parts finite
  difference methods for the wave equation}, J. Sci. Comput.,
  71 (2017), pp.~219--245.

\bibitem{Wang2018}
{\sc S.~Wang, A.~Nissen, and G.~Kreiss}, {\em Convergence of finite difference
  methods for the wave equation in two space dimensions}, Accepted in
  Math. Comp.,  (2018).

\bibitem{Wang2016b}
{\sc S.~Wang, K.~Virta, and G.~Kreiss}, {\em {High order finite difference
  methods for the wave equation with non-conforming grid interfaces}}, J. Sci. Comput., 68 (2016), pp.~1002--1028.

\end{thebibliography}

\end{document}